%% file: colt2023.tex
\documentclass[final,12pt]{colt2023} 

\usepackage{tikz}
\usepackage{answers}
\usepackage{xcolor}
\usepackage{setspace}
\usepackage{graphicx}
\usepackage[shortlabels]{enumitem}
\usepackage{multicol}
\usepackage{mathrsfs}
\usepackage{thm-restate}
\usepackage{mathtools}
\usepackage[utf8]{inputenc}
\usepackage[english]{babel}
\usepackage{natbib}
\usepackage{bbm}
\usepackage{breqn}
\usepackage{subfiles}
\usepackage{amsmath,amssymb}
\usepackage[english]{babel}
\usepackage{bm}
\usepackage{hyperref}
\usepackage{algorithm}
\newenvironment{proofsketch}{%
  \renewcommand{\proofname}{Proof Sketch}\proof}{\endproof}

\newcommand{\I}{\mathcal I}

\newcommand{\R}{\mathbb{R}}

\DeclareMathOperator*{\E}{\mathbb{E}}
\DeclareMathOperator*{\median}{median}
\newcommand{\Var}{\text{Var}}
\newcommand{\sign}{\text{sign}}

\newcommand{\eps}{\varepsilon}
\newcommand{\1}{\mathbbm{1}}

\newcommand{\abs}[1]{|#1| }
\newcommand{\clip}{\text{clip}}
\newcommand{\sym}{\text{sym}}

\newcommand{\wh}{\widehat}
\renewcommand{\hat}{\wh}
\newcommand{\wt}{\widetilde}
\newcommand{\from}{\leftarrow}
\renewcommand{\d}{\mathrm{d}}

\DeclarePairedDelimiterX{\kl}[2]{D_{KL}(}{)}{%
  #1\;\delimsize\|\;#2%
}

\hypersetup{
    colorlinks=true,
    linkcolor=blue,
    filecolor=magenta,      
    urlcolor=cyan,
}

\newcommand{\new}[1]{\textcolor{red}{#1}}
\renewcommand{\new}[1]{#1}

\title[Finite-Sample Symmetric Mean Estimation with Fisher Information Rate]{Finite-Sample Symmetric Mean Estimation\\ with Fisher Information Rate}
\usepackage{times}

\coltauthor{%
 \Name{Shivam Gupta} \Email{shivamgupta@utexas.edu}\\
 \addr The University of Texas at Austin
 \AND
 \Name{Jasper C.H. Lee} \Email{jasper.lee@wisc.edu}\\
 \addr University of Wisconsin-Madison%
 \AND
 \Name{Eric Price} \Email{ecprice@cs.utexas.edu}\\
 \addr The University of Texas at Austin%
}

\begin{document}

\maketitle

\begin{abstract}%
  The mean of an unknown variance-$\sigma^2$ distribution $f$ can be estimated from $n$ samples with variance $\frac{\sigma^2}{n}$ and nearly corresponding subgaussian rate.  When $f$ is known up to translation, this can be improved asymptotically to $\frac{1}{n\I}$, where $\I$ is the Fisher information of the distribution.  Such an improvement is not possible for general unknown $f$, but~\cite{stone1975adaptive} showed that this asymptotic convergence \emph{is} possible if $f$ is \new{\emph{symmetric}} about its mean.  Stone's bound is asymptotic, however: the $n$ required for convergence depends in an unspecified way on the distribution $f$ and failure probability $\delta$. 
  In this paper we give finite-sample guarantees for symmetric mean estimation in terms of Fisher information.  For every $f, n, \delta$ with $n > \log \frac{1}{\delta}$, we get convergence close to a subgaussian with variance $\frac{1}{n \I_r}$, where $\I_r$ is the $r$-\emph{smoothed} Fisher information with smoothing radius $r$ that decays polynomially in $n$.  Such a bound essentially matches the finite-sample guarantees in the known-$f$ setting.
\end{abstract}


\input{intro}

\input{literature}
\input{overview}

\input{proof_details}

\acks{Shivam Gupta and Eric Price are supported by NSF awards CCF-2008868, CCF-1751040 (CAREER),
and the NSF AI Institute for Foundations of Machine Learning (IFML). Some of this work was done while Shivam Gupta was visiting UC Berkeley.
Jasper C.H.~Lee is supported in part by the generous funding of a Croucher Fellowship for Postdoctoral Research, NSF award DMS-2023239, NSF Medium Award CCF-2107079 and NSF AiTF Award CCF-2006206.}

\bibliography{colt2023}

\appendix
\input{kde.tex}
\input{misc_lemmas.tex}

\input{median_of_pairwise_means.tex}
\end{document}

%% file: intro.tex
\section{Introduction}


Mean estimation is a fundamental problem in statistics.
For a distribution with variance $\sigma^2$, the empirical mean over $n$ samples has variance $\frac{\sigma^2}{n}$ and enjoys central limit behavior, asymptotically yielding error $\sigma \sqrt{2\log\frac{1}{\delta}/n}$ with failure probability $\delta$.
Substantial work~\cite{catoni,Devroye:2016,lee_valiant} has led to an estimator with a corresponding \emph{finite-sample} guarantee, achieving the same error up to a $1+o(1)$ factor.

On the other hand, consider the related problem of location estimation: if we know the exact shape of the distribution, except for an unknown translation parameter, the (asymptotic) estimation accuracy is characterized by the Fisher information.
More formally, suppose $x \sim f^{\lambda}(x) = f(x-\lambda)$ for some known $f$ but some unknown parameter $\lambda$.
The Fisher information of $f$ is defined as $\I := \E_{x \sim f}[s(x)^2]$ where $s(x)$ is the ``score'' $s(x) := f'(x)/f(x)$.
The \emph{maximum likelihood estimate} (MLE) is asymptotically normal with variance $\frac{1}{n \I}$, which is at most $\frac{\sigma^2}{n}$; and
asymptotically, the standard Cram\'er-Rao bound~\cite{cramer,rao} shows that this is optimal.


For example, the Laplace distribution has Fisher information
$\frac{2}{\sigma^2}$, and the MLE for the Laplace is the
\new{empirical} median.  Thus, for the Laplace, the \new{empirical}
median has half the asymptotic variance of the empirical mean, so it
needs half as many samples to achieve the same accuracy.  The Fisher
information can \new{sometimes} be \emph{much} larger than
$1/\sigma^2$: consider Figure~\ref{fig:gaussianmixture}, a 50-50 mixture of two Gaussians
$\frac{1}{2}N(\mu_1, \sigma_1^2) + \frac{1}{2}N(\mu_2, \sigma_2^2)$
with means $\mu_1, \mu_2 \in [-1, 1]$ and variances
$\sigma_2^2 \gg 1 \gg \sigma_1^2$.  This has variance
$\Theta(\sigma_2^2)$ and Fisher information
$\Theta(\frac{1}{\sigma_1^2})$.  Thus, the empirical mean has accuracy
proportional to the \emph{larger} standard deviation, while the MLE
has accuracy proportional to the \emph{smaller} standard deviation.
\new{In summary}, for a \emph{known} $f$ at an unknown offset $\mu$,
one can achieve an accuracy based on Fisher information, which is
never worse than the generic $\sigma^2$-dependence but can be much
better.

\begin{figure}
  {\caption{Example distributions}\label{fig:instance}}%
 {%
    \subfigure[The Laplace distribution has twice the Fisher information
        of a Gaussian with the same variance, so it can be estimated
        with asymptotically half the variance.]{%
      \includegraphics[width=.47\textwidth]{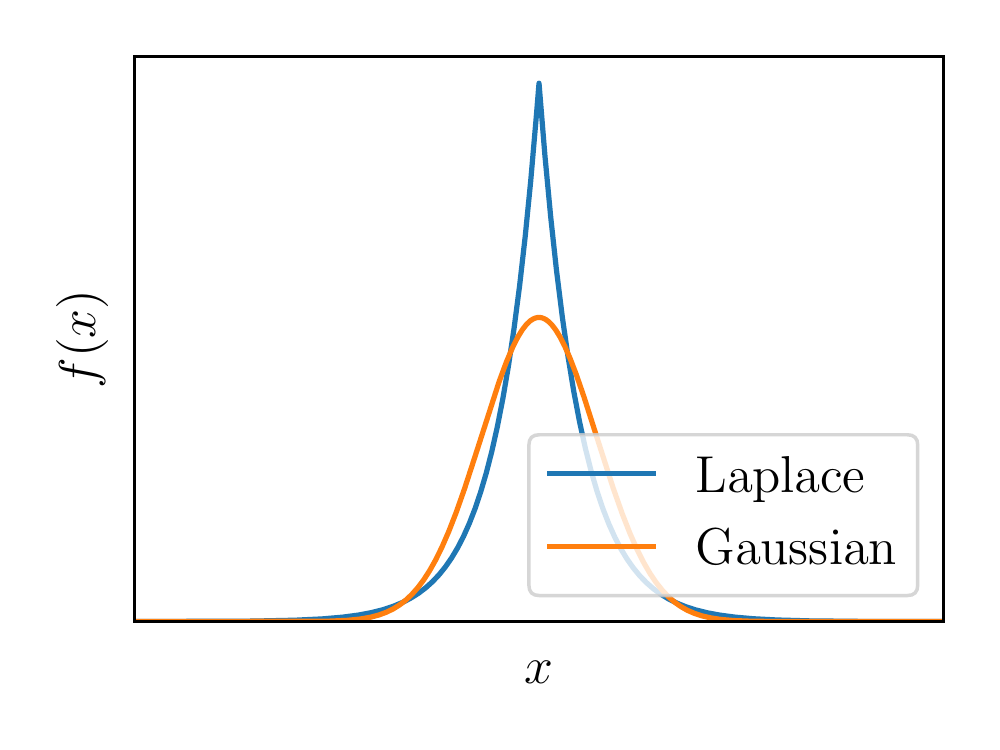}
      \label{fig:gaussianlaplace}
    }\hfill
    \subfigure[When estimating a mixture of a wide and narrow Gaussian,
       it is easier to estimate the mean of the narrow Gaussian.  When the
       distribution is known up to location, estimating this mean suffices.]{%
      \includegraphics[width=.47\textwidth]{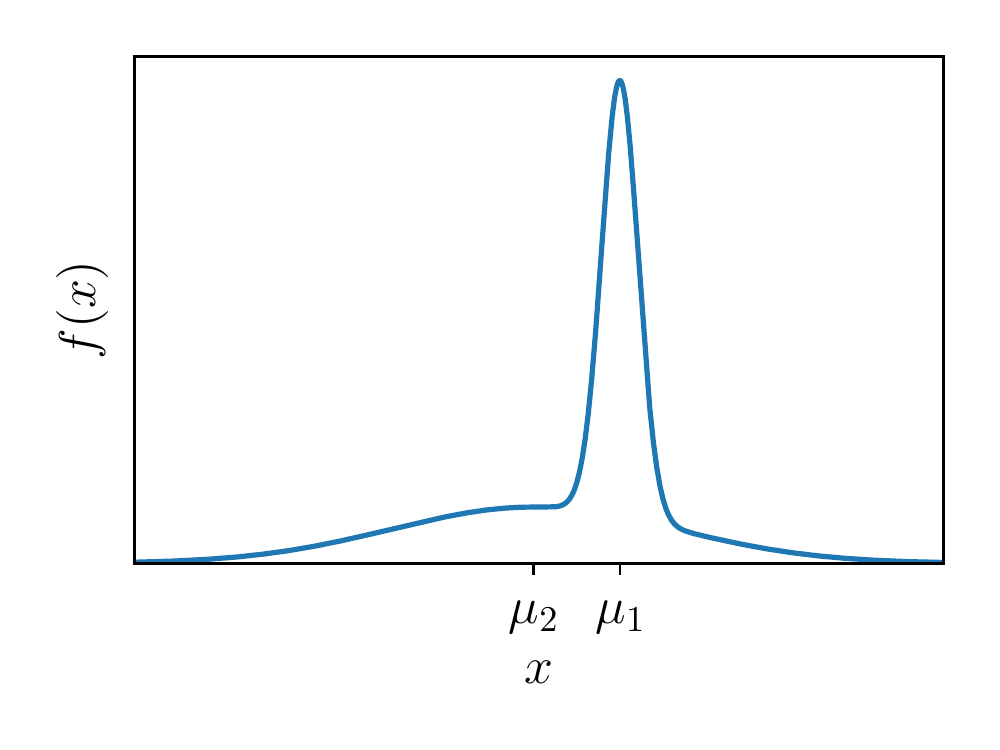}
      \label{fig:gaussianmixture}
    }
 }
\end{figure}

This poses a natural question: can we get Fisher-information--style
improvements for \emph{unknown} $f$?  Unfortunately, the answer is no.
In the mixture of Gaussians example of Figure~\ref{fig:gaussianmixture}, in the known-distribution
case we are given $\mu_2 - \mu_1$ so it suffices to estimate $\mu_1$.
This can be done with variance $\Theta(\frac{\sigma_1^2}{n})$.  In the
unknown-distribution case we need both $\mu_1$ and $\mu_2$, and estimating
$\mu_2$ induces variance $\Theta(\frac{\sigma_2^2}{n})$.
In fact, recent work has shown~\citep{Lee:2023instance} that the variance-based subgaussian error bounds are essentially instance-optimal: for \emph{every}
distribution $p$ of variance $\sigma^2$, and any $n, \delta$ with
$n \gg \log\frac{1}{\delta}$, there exists a distribution $q$ of
variance $\Theta(\sigma^2)$ where
$|\mu_p - \mu_q| = \Omega(\sigma\sqrt{\log\frac{1}{\delta}/n})$, yet
$p$ and $q$ are not distinguishable using $n$ samples with probability
$1-\delta$.


In this paper, we consider a restriction that allows for the Fisher
information benefit in mean estimation: \emph{symmetry}.  We give an
estimator that, for every \emph{symmetric} distribution $f$, estimates
its mean with an accuracy related to Fisher information.

\paragraph{Smoothed Fisher information.} To state our results, we need
the notion of \emph{smoothed} Fisher information.  One issue with the
aforementioned Fisher information results is that they are asymptotic:
the $n$ required for convergence depends on the distribution in a
possibly arbitrary way.
As one simple example, if
$f(x) = (1-\eps)N(0, 1) + \eps \delta_0$, the Fisher information is
infinite (if we see the same real-valued sample twice, that is the
exact mean) but with fewer than $1/\eps$
samples we probably only see the $N(0, 1)$ samples; here the best
estimator is the empirical mean, with error $N(0, \frac{1}{n})$.  Thus,
for finite $n$, one cannot hope for accuracy approaching the true
inverse Fisher information of a general distribution.

Recent work by~\cite{finite_sample_mle,GLP23} has given finite-$n$ bounds for the
known-distribution case in terms of the ``smoothed Fisher
information.'' For a distribution $f$, the $r$-smoothed Fisher
information $\I_r$ is the Fisher information of $f$ convolved with a
Gaussian of variance $r^2$.  In these results, $r \to 0$ as
$n \to \infty$, capturing the asymptotic behavior but giving bounds
that still apply when $f$ and $n$ vary together.

Figure~\ref{fig:gauss-sawtooth} shows an example based on adding tall but narrow ``teeth" to a standard Gaussian.  These teeth are useful for alignment \emph{within} the correct tooth, but not very useful for alignment errors that are integer multiples of the tooth width.  As a result, if the tooth width is $w$, the optimal estimator exhibits a phase
transition in its variance, with about $\frac{1}{n}$ variance for $n \ll \frac{1}{w^2}$ and $\frac{1}{n\I}$  variance for $n \gg \frac{1}{w^2}$ (see~\cite{finite_sample_mle}).
 Such a phase transition is captured by the smoothed Fisher information, which transitions at $r \approx w$.

\begin{figure}
  {\caption{Gaussian + Sawtooth}\label{fig:gauss-sawtooth}}%
  {%
    \subfigure[In the ``Gaussian+sawtooth'' example, we add tall but
    narrow ``teeth'' to a standard Gaussian.  Smoothing by radius
    larger than the width returns the distribution to nearly
    Gaussian.]{%
      \includegraphics[width=.47\textwidth]{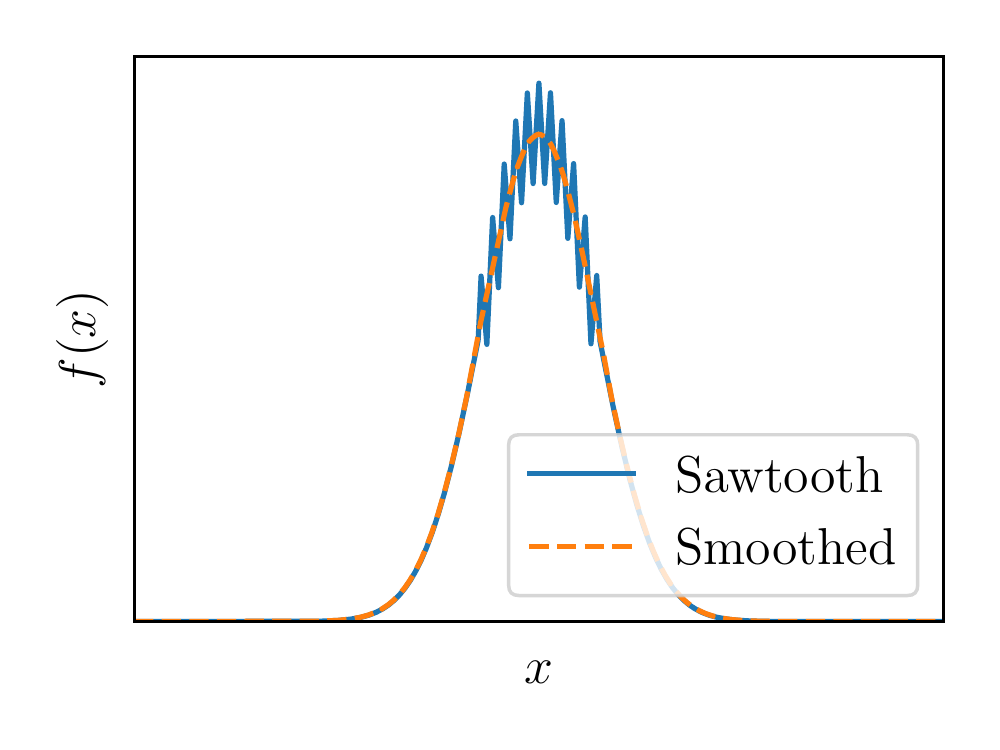}
      \label{fig:gaussian_teeth}
    }\hfill \subfigure[The smoothed Fisher information has a phase
    transition, from a large value when $r$ is small, to the standard Gaussian's 1
    when $r$ is larger than the tooth width.]{%
      \includegraphics[width=.47\textwidth]{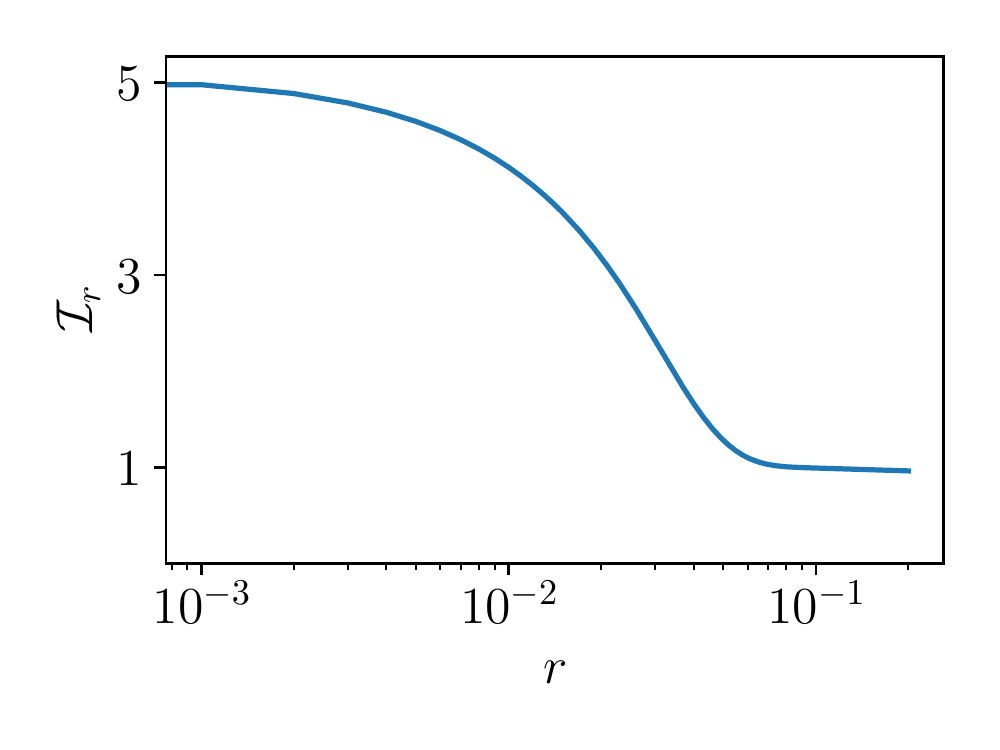}
      \label{fig:gaussian_teeth_I}
    }
 }
\end{figure}

\paragraph{Our result.}  Our main theorem is the following:

\begin{restatable}{theorem}{maintheorem}\label{thm:maintheorem}
  Let $\eta = (\frac{\log \frac{1}{\delta}}{n})^{\frac{1}{13}} < 1$, and let $\log \frac{1}{\delta} \le n/C$ for sufficiently large constant $C > 1$.  Let $f^*$ be an arbitrary symmetric distribution with variance $\sigma^2$ and mean $\mu$.  For $\eta \sigma\le r \le \sigma$, we have that the output of our estimator $\wh \mu$ satisfies
  \[
  \abs{\wh{\mu} - \mu} \leq (1 + \eta) \sqrt{\frac{2 \log \frac{2}{\delta}}{n \I_r}}
  \]
  with probability $1-\delta$.
\end{restatable}

For ``nice" distributions like the Laplace, $1/\I_r \approx 1/\I + O(r^2)$, so Theorem~\ref{thm:maintheorem} gives an error bound within $(1 + O((\frac{\log \frac{1}{\delta}}{n})^{1/13}))$ of the instance-optimal Cram\'er-Rao bound.  For other distributions, like the Gaussian+sawtooth example of Figure~\ref{fig:gauss-sawtooth}, $\I_r$ exhibits a phase transition and the error does not approach $\I$ until $n$ grows larger than some distribution-dependent quantity; in the sawtooth example, $n/\log\frac{1}{\delta}$ must be at least $O(1/w^{13})$.  As discussed above, this is qualitatively correct but with a suboptimal polynomial.

This theorem has the same form as~\cite{GLP23}, except our theorem
applies to unknown symmetric $f^*$ while theirs applies to
known, possibly asymmetric $f^*$.  The differences are (I) their
$\eps$ is a better polynomial,
$C(\frac{\log \frac{1}{\delta}}{n})^{1/10}$; and (II) their theorem
sets $r$ in terms of the interquartile range instead of standard
deviation, and so applies to infinite-variance distributions.  

Since $\frac{1}{\I_r} \leq \sigma^2 + r^2$, for appropriately chosen
$r$ our bound is never more than a $(1 + O(\eps))$-factor worse than
the subgaussian tail with variance $\frac{\sigma^2}{n}$.  This is
comparable to the results of~\cite{catoni}, although with a \new{(slightly)} weaker
convergence rate (Catoni has rate
$\eps = \frac{\log \frac{1}{\delta}}{n}$).
\new{However}, since
Theorem~\ref{thm:maintheorem} depends on the Fisher information, it can be
much better: for example, it gives a factor \new{of} $\new{2-O(\eps)}$
improvement when estimating a Laplace distribution, and variance
$\Theta\left(\frac{\min(\sigma_1^2, \sigma_2^2)}{n}\right)$ rather
than $\Theta\left(\frac{\max(\sigma_1^2, \sigma_2^2)}{n}\right)$ when
$f$ is a Gaussian mixture
$\frac{1}{2}(N(\mu, \sigma_1^2) + N(\mu, \sigma_2^2))$.

Theorem~\ref{thm:maintheorem} assumes that \new{we} are given $r$; to choose $r$ in general,
\new{we} would want a (constant-factor) estimate of $\sigma$, which can be
done if $f^*$ has bounded kurtosis.  Avoiding this dependence is an
interesting open question.

Our estimator is based on using a small fraction of samples to
construct a kernel density estimate (KDE) of $f$, then finding a
variant of the maximum likelihood estimate (MLE).  A similar approach
was used in~\cite{stone1975adaptive} to get an asymptotic bound in terms of $\I$; our
contribution is an effective bound for finite $n$ that applies to any distribution,
as well as high-probability bounds.




%% file: literature.tex
\section{Related Work}

One dimensional mean estimation is one of the most fundamental problems in statistics.
Under the assumption of finite variance, the celebrated Central Limit Theorem states that the distribution of the sample mean asymptotically convergences to a Gaussian with variance $\sigma^2/n$.
For finite-sample performance, \cite{Nemirovsky:1983,Jerrum:1986,Alon:1999} independently invented the Median-of-Means estimator, which achieves the same subgaussian concentration up to a constant factor.
A decade ago, the seminal work of \cite{catoni} initiated the search for a finite-sample subgaussian estimator with a tight multiplicative constant.
Subsequent improvements by \cite{Devroye:2016} and \cite{lee_valiant} showed how to construct a subgaussian estimator tight up to a $1+o(1)$ multiplicative factor.

This work, by contrast, assumes symmetry of the distribution about its mean.
\cite{stone1975adaptive} showed that asymptotically, the performance of mean estimation for symmetric distributions is controlled by the Fisher information instead of the variance.
Our approach is inspired by that of Stone: construct a kernel density estimate (KDE) of the underlying distribution, and perform maximum likelihood estimation (MLE) based on the KDE.
On the other hand, our bounds are explicit finite-sample bounds, and characterize the performance in terms of \emph{smoothed} Fisher information, with a smoothing radius $r$ that vanishes as $n/\log\frac{1}{\delta} \to \infty$.

Fisher information also characterizes the asymptotic error in the closely-related problem of location estimation---a parametric variant of mean estimation---under the much stronger assumption that we know the shape of the entire distribution up to some unknown translation~\cite{vanDerVaart:2000asymptotic}.
The recent works by Gupta et al.~\citeyearpar{finite_sample_mle,GLP23} developed a finite-sample theory of location estimation with error in terms of the smoothed Fisher information, up to a $1+o(1)$ factor.
Our algorithm also draws from the techniques in this line of work.
In particular, instead of finding the maximum of the empirical log-likelihood function, they perform a single step of Newton's method to approximate a root of the derivative. 
This modification both simplifies the algorithmic implementation and yields analysis advantages.
Our algorithm and analysis crucially leverage the same simplified view of the MLE.

Under the assumption of symmetry, therefore, our estimator is almost as good as the best estimator for \emph{known} distributions.  
In the classical asymptotic theory, this is referred to as an ``adaptive estimator"~\cite{bickel_adaptive_estimation}.  Such estimators are known for other semi-parametric problems, such as linear regression with symmetric errors, estimation of the parameters of an elliptical distribution, and two-sample comparison of a known density. A natural question for future work is to get a finite-sample theory for these problems.

The statistics and computer science communities have also been actively studying the high-dimensional mean estimation problem.
\cite{Lugosi:2019subgaussian} proposed the first subgaussian high-dimensional mean estimator up to a multiplicative constant, but with exponential time.
\cite{Hopkins:2020} and \cite{cherapanamjeri19b} later improved the result to take quadratic time.  A tight constant factor was achieved by~\cite{Lee:2022VeryHigh} in the ``very high-dimensional" regime, but it remains an open problem to achieve a subgaussian estimator with tight constants in general.

Recent years have seen a surge of interest in using maximum likelihood in theoretical computer science, as a generic algorithm that can give efficient guarantees.  Such papers include, for example, 
profile maximum likelihood for distribution testing and functional estimation~\cite{Acharya:2011competitive,Acharya:2017PML,Hao:2019broad,Pavlichin:2019approximatePML,Charikar:2019efficient,Anari:2020instance}; 
space-efficient streaming algorithms~\cite{pettie2021information}; and other statistical estimation problems~\cite{daskalakis2018efficient,vinayak2019maximum,awasthi:2022trimmed}.


The result of this work has an ``instance optimal" flavor: for each distribution, the error bounds are phrased in terms of the (smoothed) Fisher information.
The Cram\'{e}r-Rao bound shows that, even if we knew the distribution shape, we cannot hope to do better than the Fisher information bounds.
Instance optimality and related notions have also been studied in the context of other statistical problems, for example, identity testing~\cite{Valiant:2017}, learning discrete distributions~\cite{Valiant:2016}, mean estimation without symmetry~\cite{Lee:2023instance} and differentially-private mean estimation~\cite{Asi:2020Arxiv,Asi:2020NeurIPS,Huang:2021}.

%% file: overview.tex
\section{Proof Sketch}\label{sec:proofsketch}

In this section we give a very high-level overview of our proof
approach; for a more detailed quantitative overview, see
Section~\ref{sec:proofdetails}.
\new{Here,} we will describe how to use $(1 + O(\eta)) n$ samples
to get accuracy
$(1 + O(\eta)) \sqrt{\frac{2 \log \frac{2}{\delta}}{n}}$ with
probability $1-\delta$, for $\eta = (\log\frac{1}{\delta}/n)^{1/13}$; rescaling parameters gives the result.

Our algorithm proceeds in two phases.  In the first phase, we use a
small number of samples (namely $\eta n$) to produce an initial estimate
$\mu_1$ of $\mu$, and an approximation $\wh{f}_r$ to $f_r$.  Since $f$
is symmetric, we can use the median of pairwise means
estimator~\cite{mintonprice}:
$\mu_1 = \median_{i \in [\eta n / 2]} \frac{x_{2i-1} +
  x_{2i}}{2}$. This has subgaussian tails corresponding to the
variance of $f$:
\[
  \eps := \mu_1 - \mu \text{~~~satisfies~~~}  \abs{\eps} \lesssim  \sqrt{\frac{\sigma^2 \log \frac{2}{\delta}}{\eta n}}
\]
with probability $1-\delta$, for every $\delta > 0$.  In the second
stage, we want to refine this estimate to $(1 + O(\eta))\sqrt{\frac{\log \frac{2}{\delta}}{n\I_r}}$ error, which is a small polynomial factor better
(by at least $\sqrt{\eta}$, but perhaps even better, like $(\frac{n}{\log \frac{1}{\delta}})^{0.1}$).  We do so with,
essentially, one step of Newton's method.

\paragraph{Background: known distribution case.}  Suppose we knew the
distribution of $f_r$, except for the location shift.
We \new{consider} centering the distribution at $\mu_1 = \mu + \eps$, \new{i.e.~define}
$\wt{f}_r(x) = f_r(x - \eps)$, \new{in order to estimate $\eps$}. 
\new{To do so,} take the score
function
\[
  \wt{s}_r(x) := \frac{\wt{f}'_r(x)}{\wt{f}_r(x)}
\]
which satisfies $\wt{s}_r(x+\eps) = s_r(x)$, where $s_r$ is the score
function of $f$.  Therefore, by standard properties of the score
function,
\begin{align*}
  \E_{x \sim f_r}[\wt{s}_r(x + \eps)] &= 0\\
  \E_{x \sim f_r}[-\tilde{s}_r'(x+\eps)] = \E_{x \sim f_r}[\wt{s}_r^2(x+\eps)] &= \I_r.
\end{align*}
Since we know $\wt{s}_r$, we can take our $n$ samples $x_i$, add fresh
independent noise $w_i \sim N(0, r^2)$ to get $x_i + w_i \sim f_r$,
and compute the empirical average
\begin{align*}
  \hat{\E}[\wt{s}_r(x_i + w_i)] := \frac{1}{n} \sum_{i=1}^n \wt{s}_r(x_i + w_i) 
\end{align*}
One can show that this concentrates, so by a Taylor approximation
\begin{align}
  \hat{\E}[\wt{s}_r(x_i + w_i)] \approx \E_{x \sim f_r}[\wt{s}_r(x)] \approx \E_{x \sim f_r}[\wt{s}_r(x + \eps) - \eps \tilde{s}_r'(x+\eps)] = - \eps \E_{x \sim f_r}[\tilde{s}_r'(x+\eps)] = \eps \I_r\label{eq:knowndist}
\end{align}
Thus we can estimate $\mu$ as
\[
  \wh{\mu} := \mu_1 - \I_r^{-1} \hat{\E}[\wt{s}_r(x_i + w_i)] \approx \mu_1 - \eps = \mu .
\]
The new estimate $\wh{\mu}$ has error only from the two approximations
in~\eqref{eq:knowndist}: (I) how well the empirical average score
\new{concentrates to} the true average, and (II) the Taylor approximation.

At $\eps = 0$, error (II) is zero and error (I) has variance
$\frac{1}{n}\Var(s_r(x)) = \frac{\I_r}{n}$.  Since $\wh{\mu}$ rescales
by $\I_r^{-1}$, this means $\wh{\mu}$ has variance $\frac{1}{n\I_r}$
at $\eps = 0$---precisely the Cram\'er-Rao bound we want to achieve.
It was shown \new{by}~\cite{finite_sample_mle} that the same bound holds to within a $1+o(1)$ factor as long as $\eps$ is
small relative to $r$ (namely, $\abs{\eps} \ll r^2 \sqrt{\I_r}$), and that
the error satisfies a subgaussian tail bound matching this variance.

\paragraph{Our setting: unknown distribution case.}  The above
algorithm \new{for the known-distribution case} uses knowledge of the distribution in two ways: to compute
$\wt{s}$, and to estimate $\I_r$ to rescale it.  But what happens if
we use some function $g(x)$ other than the score?  If $g$ is
antisymmetric about $\mu_1$, we still have
$\E_{x \sim f_r}[g(x+\eps)] = 0$, \new{and so}
\begin{align}
  \hat{\E}[g(x_i + w_i)] \approx \E_{x \sim f_r}[g(x)] \approx \E_{x \sim f_r}[g(x+\eps) - \eps g'(x)] = -\eps \E_{x \sim f_r}[g'(x)]\label{eq:pickg}
\end{align}
If $g$ is reasonably smooth and $\eps$ is small relative to $r$, the Taylor approximation will be quite good, in which case
\begin{align}\label{eq:gerr}
  \wh{\eps} := -\frac{\hat{\E}[g(x_i + w_i)]}{\E_{x \sim f_r}[g'(x)]}
\end{align}
is a low-bias estimator of $\eps$.  For small $\eps$, we expect this
estimator to have variance close to the variance at $\eps =0$, which
is
\begin{align}
\Var(\wh{\mu}) \approx \frac{1}{n} \frac{\E_{\new{f_r}}[g^2(y)]}{\E_{\new{f_r}}[g'(y)]^2}.\label{eq:newvariance}
\end{align}


One can show that this variance is at least $\new{\frac{1}{n\I_r}}$,
with minimum achieved when $g$ is the score function
$s_r(x) = \frac{f_r'(x)}{f_r(x)}$, matching the Cram\'er-Rao bound; see Proposition~\ref{prop:score_is_best_g} at the end of the section.
But this argument is fairly robust: we just need $g(x)$ to be an
antisymmetric function that approximates $s_r$ well under these two
expectations, and that is robust to perturbations $\eps \ll r$.

Our algorithm then is: using our initial set of $\eta n$ samples in
the first stage, we compute the kernel density estimate (KDE)
\[
  \wh{f}_r(x) := \frac{1}{n}\sum_{i=1}^n \phi\left(\frac{x - x_i}{r}\right)
\]
where $\phi(t)$ is the Gaussian density
$\phi(t) = \frac{1}{\sqrt{2 \pi}}e^{-t^2/2}$.  This has corresponding
score function $\wh{s}_r(x) = \frac{\wh{f}'_r(x)}{\wh{f}_r(x)}$.  We first
clip the score function to have magnitude at most
$T \approx \frac{\sqrt{\log n}}{r}$, and then antisymmetrize this
score function about $\mu_1$ by just copying the right side over: setting
$\wh{s}^{sym}_r(x) = \wh{s}_r^{clip}(2 \mu_1 - x)$ for $x \leq \mu_1$.
This produces the antisymmetric function $\wh{s}^{sym}_r$ we use as
$g$ in the above proof outline.

The final step in our algorithm is that, in order to estimate our
target via~\eqref{eq:gerr}, we need to approximate
$ \E_{x \sim f_r}[\wh{s}^{sym'}(x)]$.  Since $\wh{s}^{sym}$ is close to the true score $s_r(x)$, this value is close (within $1 + O(\eta)$) to $\I_r$.  
Thus, we can just make an estimate $\wh{\I}_r$ of $\I_r$ using the distribution $\wh{f}_r$.


Our sources of error are the following: (I) the empirical
\new{concentration} to the expectation of $\wh{s}^{sym}(x)$; (II) the Taylor
approximation in~\eqref{eq:pickg}; (III) the increase in
variance~\eqref{eq:newvariance} due to $\wh{s}^{sym}$ not being the
exact score; and (IV) error from approximating
$\E_{x \sim f_r}[\wh{s}^{sym'}(x)]$ by
$\wh{\I}_r$ in~\eqref{eq:gerr}.

Unlike in the known-distribution case, clipping is necessary for bounding
error (I).  The true score concentrates in expectation because
$s_r(x)$ is subgamma over $x \sim f_r$.  However, $\wh{s}_r(x)$ may not be
so concentrated.  
Consider the example
$f(x) = (1 - \frac{2}{n})N(0, 1) + \frac{1}{n} \delta_{-\sqrt{n}} +
\frac{1}{n} \delta_{\sqrt{n}}$.  In this example, $x \sim f$ is usually constant but has a
$\Theta(\frac{1}{n})$ chance of being quite large; in this case, it is
likely that the large points will not appear for the KDE but will
appear exactly once for the second stage.  The KDE then gives them
large scores (about $\sqrt{n} / r$), leading to excessive final error
($\Theta(\frac{1}{r \sqrt{n}})$ not $\Theta(\frac{1}{\sqrt{n}})$).
Once the scores are clipped, however, we can bound the error (I) with
high probability via Bernstein's inequality.  The clipping threshold $T$ is large enough to have negligible effect on the expectations (III-IV); since the true score is subgaussian, with high probability it is not clipped.
\new{Specifically, in the Gaussian + Symmetric Dirac Deltas example above , the ``excess error" in the above constant probability event is now $O(\sqrt{\log \eta n}/(r n))$.
Recalling that $r \approx 1/n^{1/13}$ in Theorem~\ref{thm:maintheorem}, the excess error after clipping is $\ll O(1/\sqrt{n})$.}

Error (II) is bounded when $\eps$ is small in a similar manner to
previous work \new{in the known distribution case}.  For errors (III) and (IV), we just need to show that,
with high probability, our KDE $\wh{f}_r \approx f_r$ and
$\wh{s}^{sym}_r \approx s_r$, in different metrics but all in expectation
over $\wh{f}$.

\paragraph{Comparison to~\cite{stone1975adaptive}.}  Our approach, of taking an initial estimate and KDE and refining it with one Newton step, is similar to~\cite{stone1975adaptive}.  The main difference is that our work needs more careful bounds:~\cite{stone1975adaptive} shows convergence in probability to $N(0, \frac{1}{n\I})$, which requires fixing the distribution $f$ and failure probability $\delta$ before sending $n \to \infty$.   By separating the distribution dependence into $\I_r$, we can express and prove bounds for any $f, n, \delta$.

We end this section with a short proof relating~\eqref{eq:newvariance} to the score and Fisher information.
\begin{proposition}
\label{prop:score_is_best_g}
For every antisymmetric function $g$ that is continuously differentiable and whose derivative $g'$ is integrable under $f_r$, we have
\[\frac{\E_{{f_r}}[g^2(y)]}{\E_{{f_r}}[g'(y)]^2} \ge \frac{1}{\I_r} \]
with equality achieved when $g(y) = s_r(y) = f'_r(y)/f_r(y)$.
\end{proposition}

\begin{proof}
First, observe that by integration by parts, we have
\[
\E_{f_r}[g'(y)] = \int_{\R} f_r(y) g'(y) \, dy =  [f_r(y) g(y)]_{-\infty}^{\infty} - \int_{\R} f'_r(y) g(y) \, dy = - \int_{\R} f'_r(y) g(y) \, dy 
\]
where the last equality is by the symmetry of $f$ and antisymmetry of $g$.  
Furthermore,
\[
   \int_{\R} f'_r(y) g(y) \, dy =  \int_{\R} \frac{f'_r(y)}{f_r(y)} g(y) f_r(y) \, dy = \E_{f_r}[s_r(y) g(y)]
\]
which means
\[
\E_{f_r}[g'(y)]^2 = \E_{f_r}[s_r(y) g(y)]^2 \leq \E_{f_r}[s_r(y)^2]\E_{f_r}[g(y)]^2 = \I_r \E_{f_r}[g(y)]^2
\]
by Cauchy-Schwarz, with equality achieved when $g(y) = s_r(y)$.
\end{proof}

%% file: proof_details.tex
\section{Key Steps in Proof}\label{sec:proofdetails}
Here, we highlight the key steps of our proof. For the full proofs, see the Appendix.
\paragraph{Notation.} Let $f^*$ be an arbitrary symmetric distribution with mean $\mu$ and variance $\sigma^2$, and let $f_r$ be the $r$-smoothed version of $f^*$. Let $s_r$ be the score function of $f_r$, so that $s_r(x) = \frac{f_r'(x)}{f_r(x)}$. Let $\I_r = \E_{x \sim f_r}\left[ s_r(x)^2\right] = -\E_{x \sim f_r}\left[s_r'(x) \right]$ be the Fisher information of $f_r$.

Let $w_r$ be the density function of $\mathcal N(0, r^2)$. Then, the Kernel Density Estimate (KDE) $\wh f_r$ from $N$ samples $Y_1, \dots Y_N \sim f^*$ is given by
\begin{equation}
    \label{eq:kde_def_main_body} 
    \wh f_r(x) = \frac{1}{N} \sum_{i=1}^N w_r(x - Y_i)
\end{equation}
It has score function $\wh s_r$ with $\wh s_r(x) = \frac{\wh f_r'(x)}{\wh f_r(x)}$. Let $\wh s_r^\clip$ be the clipped KDE score from $N$ samples with associated failure probability $\delta$, given by
\begin{equation}
    \label{eq:clipped_kde_def_main_body}
    \wh s_r^\clip(x) = \sign(\wh s_r(x)) \cdot \min\left(|\wh s_r(x)|, \frac{2}{r} \sqrt{\log \frac{N}{\log \frac{1}{\delta}}} \right)
\end{equation}
Define  the symmetrized clipped KDE score $\wh s_r^\sym$ from $N$ samples, symmetrized around $y$, as
\begin{equation}
    \label{eq:sym_kde_def_main_body}
    \wh s_r^\sym(x) =  
    \begin{cases}
        \wh s_r^\clip(x) & x \ge y\\
        -\wh s_r^\clip(2y - x) & x < y
    \end{cases}
\end{equation}

In what follows, we first analyze the clipped KDE score (Section~\ref{sec:clipped}), before showing that symmetrizing it at a $\mu+\eps$ for small $|\eps|$ does not add too much error (Section~\ref{sec:sym}).
Using similar techniques, we prove that $\I_r$ can be computed directly from the KDE (Section~\ref{sec:Ir}).
Section~\ref{sec:local} then analyzes the Newton step of the estimation, and finally Section~\ref{sec:global} assembles all the guarantees into Lemma~\ref{lem:global_estimation_empirical_mean}, from which our main result Theorem~\ref{thm:maintheorem} follows as a corollary.

\subsection{Clipped KDE Score}
\label{sec:clipped}

We first show that the clipped KDE score $\wh s_r^\clip$ approximates the true score $s_r$ in an $\ell_2$ sense:
\begin{restatable}[Clipped KDE score error]{lemma}{clippedkdeerror}
\label{lem:clipped_kde_error}
    Let $\wh s_r^\clip$ be the clipped Kernel Density estimate from $N$ samples, defined in \eqref{eq:clipped_kde_def}.
     Let $\gamma > C$ be a parameter, for large enough constant $C \ge 1$. Then for any $r \le \sigma$ and $\frac{N}{\log \frac{1}{\delta}} \ge \left(\frac{\gamma^{5/12} \sigma }{r} \right)^{6 + \beta}$ for $\beta > 0$, with probability $1 - \delta$, we have that,
    \[
        \E_{x \sim f_r}\left[ (\wh s_r^\clip(x) - s_r(x))^2\right] \lesssim \frac{\I_r}{\gamma}
    \]
    This holds even for asymmetric $f^*$ and $f_r$.
\end{restatable}
\begin{proofsketch}
   We refer to the radius $t\sigma_r$ region around the true mean of $f^*$ as the ``typical region'', and to the region with density at least $\alpha= \frac{1}{t^3 \sigma_r}$ as the ``large density region''.
    We break up the expectation above into $3$ parts: (I) the typical, large density region, (II) the typical, small density region, and (III) the atypical region. We then bound the expectation in each of these regions individually. 

    To bound the expectation in regions (II) and (III), observe that both regions (II) and (III) have total probability at most $O\left(\frac{1}{t^2} \right)$. For our clipping threshold, we show that the expectation of $\wh s_r^\clip(x)^2$ and $s_r$ on a region with this probability is bounded by $O\left(\frac{\I_r}{\gamma} \right)$. 
    
    For region (I), we employ a binning argument. We first show that if we fix $x$ with $f_r(x) \ge \alpha$, then, for small enough $\eps$ and for all $|\zeta| \le |\eps|$, with probability $1-\delta$, $\wh s_r(x+\zeta)$ approximates $s_r(x+\zeta)$ up to error depending on $\eps, \delta, \alpha$ and $N$. That is, our KDE score approximates the true score well within bins of size $\eps$ with probability $1 - \delta$. Then, by union bounding over $O\left(\frac{t \sigma_r}{\eps}\right)$ bins, for appropriately chosen $\eps$, we show that, with probability $1 - \delta$ for all $x$ in region (I), $|\wh s_r^\clip(x) - s_r(x)| \lesssim \sqrt{\frac{\I_r}{\gamma}}$ so that the expectation of $(\wh s_r^\clip(x) - s_r(x))^2$ in region (I) is bounded by $O\left(\frac{\I_r}{\gamma} \right)$.

    Putting our bounds together then gives the claim.
\end{proofsketch}

\subsection{Symmetrization}
\label{sec:sym}
This section shows that $\wh s_r^\sym$ symmetrized at $\mu + \eps$ for small $\eps$ has mean $\approx \eps \I_r$ and variance $\approx \I_r$.
\begin{restatable}[Symmetrized Clipped KDE score variance]{lemma}{symkdevariancecor}
\label{cor:sym_kde_variance}
   Let $\wh s_r^\sym$ be the symmetrized clipped Kernel Density Estimate score from $N$ samples, symmetrized around $\mu + \eps$ for $|\eps| \le r/60$, as defined in \eqref{eq:symmetrized_kde_def}. Let $\gamma > C$ be a parameter for large enough constant $C$. Then for any $r \le \sigma$ and $\frac{N}{\log \frac{1}{\delta}} \ge \left(\frac{\gamma^{5/12} \sigma }{r} \right)^{6 + \beta}$ for $\beta > 0$, if $|\eps| \le r^2 \sqrt{\frac{\I_r}{\gamma}}$, with probability $1 - \delta$,
    \[
        |\E_{x \sim f_r}\left[\wh s_r^\sym(x)^2 \right] - \I_r| \lesssim \frac{\I_r}{\sqrt{\gamma}}
    \]
\end{restatable}
\begin{proofsketch}
    First, we show that
    \[
        \E_{x \sim f_r}\left[(\wh s_r^\sym(x) - s_r(x))^2 \right] \lesssim \E_{x \sim f_r}[(\wh s_r^\clip(x) - s_r(x))^2]
    \]
    so that by the previous Lemma~\ref{lem:clipped_kde_error}, it's bounded by $O\left(\frac{\I_r}{\gamma} \right)$. Then, we can show the claim using the triangle inequality in $\ell_2$.
\end{proofsketch}

The Taylor approximation~\eqref{eq:pickg} leads to:

\begin{restatable}[Symmetrized Clipped KDE score mean]{lemma}{symkdeerror}
\label{lem:sym_kde_mean}
    Let $\wh s_r^\sym$ be the symmetrized clipped Kernel Density Estimate score from $N$ samples, symmetrized around $\mu + \eps$ for $|\eps| \le r/60$, as defined in \eqref{eq:symmetrized_kde_def}. Let $\gamma > C$ be a parameter for large enough constant $C$. Then for any $r \le \sigma$ and $\frac{N}{\log \frac{1}{\delta}} \ge \left(\frac{\gamma^{5/12} \sigma }{r} \right)^{6 + \beta}$ for $\beta > 0$, if $|\eps| \le r^2 \sqrt{\frac{\I_r}{\gamma}}$, with probability $1 - \delta$,
    \[
        \left|\E_{x \sim f_r}\left[\wh s_r^\sym(x) \right] -  \eps \I_r\right| \lesssim \frac{\eps \I_r}{\sqrt{\gamma}}
    \]
\end{restatable}

\subsection{Estimating $\I_r$}
\label{sec:Ir}
To perform a step of Newton's method, we need an estimate of the Fisher information $\I_r$. We show that $\wh \I_r = \E_{x \sim \wh f_r}\left[ \wh s_r^{\sym} (x)^2\right]$ is a good estimate whenever $\wh s_r^\sym$ satisfies the conditions above.
\begin{restatable}[Smoothed Fisher information Estimation]{lemma}{fisherinfoestimation}
\label{lem:fisher_information_estimation}
      Let $\gamma \ge C$ for large constant $C \ge 1$ be a parameter. Suppose we have a function $\Tilde s_r$ that satisfies for $r \le \sigma$
    \[
        \left|\E_{x \sim f_r}[\Tilde s_r(x)^2] - \I_r \right| \lesssim \frac{\I_r}{\sqrt{\gamma}}
    \]
    and that $|\Tilde s_r(x)| \le \frac{2}{r} \sqrt{\log \frac{N}{\log \frac{1}{\delta}}}$ for all $x$. Let $\wh f_r$ be the kernel density estimate for $f_r$ from $N$ samples, as defined in \eqref{eq:kde_def}.
    Then, for $\frac{N}{\log \frac{1}{\delta}} \ge \left( \gamma^{5/12} \frac{\sigma}{r}\right)^{6 + \beta}$ for some small constant $\beta > 0$, with probability $1 - \delta$, we have
    \[
        \left|\E_{x \sim \wh f_r}[\Tilde s_r(x)^2] - \I_r \right| \lesssim \frac{\I_r}{\sqrt{\gamma}}
    \]
\end{restatable}
\begin{proofsketch}
    We have
    \begin{align*}
        \E_{x \sim \wh f_r}\left[\Tilde s_r(x)^2 \right] = \E_{x \sim f_r}[\Tilde s_r(x)^2] + \int_{-\infty}^\infty \left(\wh f_r(x) - f_r(x) \right) \Tilde s_r(x)^2 dx
    \end{align*}
    As in the proof of Lemma~\ref{lem:clipped_kde_error}, we again break up the integral above into $3$ parts and bound each part separately by $O\left(\frac{\I_r}{\sqrt{\gamma}} \right)$. Finally, we make use of our assumption that $\E_{x \sim f_r}[\Tilde s_r(x)^2] \approx \I_r$ along with our bound on the integral to show the claim.
\end{proofsketch}

To conclude, we have with high probability  ($1 - \frac{\delta}{\xi}$) that our KDE satisfies the following:

\begin{restatable}[KDE Estimation Properties]{property}{localestimationassumption}
\label{property:local_estimation_empirical_mean}
Let $f_r$ be the $r$-smoothed version of symmetric distribution $f^*$ in Algorithm~\ref{alg:local_estimation_empirical_mean}, with Fisher information $\I_r$. For parameters $\gamma > C$ for some sufficiently large constant $C$ and $\xi$, $\wh s_r^\sym$ satisfies that for symmetrization point $\mu_1 = \mu + \eps$,
\begin{align*}
    \left|\E_{x \sim f_r}\left[\wh s_r^\sym(x) \right]  - \eps \I_r\right| \lesssim \frac{\eps \I_r}{\sqrt{\gamma}} \quad \text{and} \quad \left|\E_{x \sim f_r}\left[\wh s_r^\sym(x)^2 \right] - \I_r \right| \lesssim \frac{\I_r}{\sqrt{\gamma}}
\end{align*}
   and $|\wh s_r^\sym(x)| \le \frac{2}{r} \sqrt{\log \frac{n}{\xi\log \frac{\xi}{\delta}}}$ for all $x$. Furthermore, the Fisher information estimate $\wh \I_r$ satisfies
   \[
        \left|\wh \I_r - \I_r\right| \lesssim \frac{\I_r}{\sqrt{\gamma}}
   \]
\end{restatable}

\subsection{Local Estimation}
\label{sec:local}
We then show that Property~\ref{property:local_estimation_empirical_mean} implies that Algorithm~\ref{alg:local_estimation_empirical_mean}, which  does one approximate Newton step, gets high accuracy.

\begin{algorithm}[H]\caption{Local Estimation}
\label{alg:local_estimation_empirical_mean}
\vspace*{3mm}
\paragraph{Input Parameters:}
\begin{itemize}
\item $n$ samples $x_1, \dots, x_{n} \sim f^*$, the symmetrized and clipped KDE score function $\wh s_r^\sym$, symmetrization point $\mu_1$, Fisher information estimate $\wh \I_r$
\end{itemize}
\begin{enumerate}
\item For each sample $x_i$, compute a perturbed sample $x'_i = x_i + \mathcal N(0,r^2)$ where all the Gaussian noise are drawn independently across all the samples.
\item Compute $\hat \eps = \frac{1}{\wh \I_r n}\sum_{i=1}^{n} \wh s_r^\sym(x_i')$. Return $\hat \mu = \mu_1 - \hat \eps$.
\end{enumerate}
\end{algorithm}

\begin{restatable}[Local Estimation]{lemma}{localestimationlemma}
\label{thm:local_estimation}
    In Algorithm~\ref{alg:local_estimation_empirical_mean}, let $f_r$ be the $r$-smoothed version of symmetric distribution $f^*$, with score function $s_r$ and Fisher information $\I_r$. Suppose for parameters $\gamma, \xi$, and symmetrized clipped KDE score $\wh s^\sym$ symmetrized around $\mu_1$, Property~\ref{property:local_estimation_empirical_mean} is satisfied. Then, with probability $1 - \delta$, the output $\hat \mu$ of Algorithm~\ref{alg:local_estimation_empirical_mean} satisfies
   \[
   \left| \hat \mu - \mu \right| \le \left(1 + O\left(\frac{1}{\sqrt{\gamma}}\right) \right)\sqrt{\frac{2 \log \frac{2}{\delta}}{n \I_r}}  + O\left(\frac{ \sqrt{\log \frac{n}{\xi\log \frac{\xi}{\delta}}}}{r\I_r} \cdot \frac{\log \frac{2}{\delta}}{n}\right)+ O\left(\frac{\eps}{\sqrt{\gamma}}\right)
   \]
\end{restatable}
\begin{proofsketch}
    We bound $\wh{\mu} - \mu$ using~\eqref{eq:pickg} for $g(x) = \wh s_r^\sym(x)$.
    We apply Bernstein's inequality to concentrate $\wh{\E}_{x \sim f_r}[\wh s_r^\sym(x)]$.  The first term in our bound is the variance term and the second is the exponential term.  The final term in our error bound comes from the difference between $\E[\wh{s}_r^\sym(x)]$ and $\eps \I_r$ bounded by Property~\ref{property:local_estimation_empirical_mean}.
\end{proofsketch}

\subsection{Global Estimation}
\label{sec:global}
In order to perform our final estimation, we compute an initial estimate $\mu_1$ of $\mu$, and our KDE $\wh f_r$ with associated symmetrized clipped score $\wh s_r^{\sym}$ around $\mu_1$, along with our Fisher information estimate $\wh \I_r$. We combine these with our Local Estimation algorithm to obtain our final estimate $\wh \mu$. Lemma~\ref{lem:global_estimation_empirical_mean} shows our formal guarantee on the performance of our final estimate $\hat \mu$. Then, 

For our initial estimate $\mu_1$, we make use of the median of pairwise means estimator for symmetric distributions, implied by \cite{mintonprice}.

\begin{restatable}[Median of pairwise means estimator]{lemma}{medianofmeanslemma}
\label{lem:median_of_means}
   Let $X_1, X_2, \dotsc, X_n$ be drawn from a symmetric distribution with mean $\mu$ and variance $\sigma^2$.  For every constant $C_1 > 0$ there exists a constant $C_2$ such that
   $
   \wh{\mu} := \median_{i \in [n/2]} \frac{X_{2i-1} + X_{2i}}{2}
   $
   satisfies
   \[
           \abs{\wh{\mu} - \mu} \leq C_2\sigma \cdot \sqrt{\frac{\log \frac{2}{\delta}}{n}}
   \]
   with probability $1-\delta$, for all $\delta$ with $\log \frac{1}{\delta} \leq C_1 n$.
\end{restatable}

\begin{algorithm}[H]\caption{Global Estimation}
    \label{alg:global_estimation_empirical_mean} 
    \vspace*{3mm}
    \paragraph{Input parameters:}
    \begin{itemize}
        \item Failure probability $\delta$, Samples $x_1, \dots, x_n \sim f^*$, smoothing parameter $r$, approximation parameter $\xi > 0$.
    \end{itemize}
    \begin{enumerate}
        \item First, use the first $n/\xi$ samples to compute an initial estimate $\mu_1$ of the mean $\mu$ by using the Median-of-pairwise-means estimator in Lemma~\ref{lem:median_of_means}.
        \item Use the next $n/\xi$ samples to compute the kernel density estimate $\wh f_r$ of $f_r$ (as defined in \eqref{eq:kde_def}), along with the associated symmeterized, clipped KDE score $\wh s_r^\sym$ (as defined in \eqref{eq:symmetrized_kde_def}), clipped at $\frac{2}{r} \sqrt{\log \frac{n}{\xi \log \frac{\xi}{\delta}}}$ and symmetrized around the initial estimate $\mu_1$. Compute the Fisher information estimate $\wh \I_r = \E_{x\sim \wh f_r}\left[\wh s_r^\sym(x)^2 \right]$.
        \item Run Algorithm~\ref{alg:local_estimation_empirical_mean} using the remaining $n - \frac{2n}{\xi}$ samples, and return the final estimate $\hat \mu$.
    \end{enumerate}
\end{algorithm}

\begin{restatable}[Global Estimation]{lemma}{globalestimationlemma}
\label{lem:global_estimation_empirical_mean}
Let $ \xi > C$ for large enough constant $C > 3$ be a parameter, and suppose $\xi \le \gamma \le \left(\frac{n}{\xi\log \frac{1}{\delta}} \right)^{2/5-\alpha}$ for constant $\alpha > 0$. For any $r \le \sigma$ and $\frac{n}{\log \frac{1}{\delta}} \ge \xi \left(\frac{\gamma^{5/12} \sigma }{r} \right)^{6+\alpha}$, with probability $1 - \delta$, Algorithm~\ref{alg:global_estimation_empirical_mean} outputs an estimate $\hat \mu$ with
\[
    |\hat \mu - \mu| \le \left(1 + O\left( \frac{1}{\sqrt{\gamma}}\right) + O\left( \frac{1}{\xi}\right) \right) \sqrt{\frac{2\log \frac{2}{\delta}}{n \I_r}}+ O\left(\frac{\sigma}{r \sqrt{\gamma}}\sqrt{\frac{\xi \log \frac{\xi}{\delta}}{n \I_r}}\right)
\]
\end{restatable}
\begin{proofsketch}
Combining Lemmas~\ref{lem:median_of_means},~\ref{lem:sym_kde_mean},~\ref{cor:sym_kde_variance},~\ref{lem:fisher_information_estimation}, we know that $\wh{s}_r^\sym(x)$ and $\wh{\I}_r$ computed in Steps 1 and 2 satisfy Property~\ref{property:local_estimation_empirical_mean} with high probability.
Invoking Lemma~\ref{thm:local_estimation} on Step 3 yields the result.
\end{proofsketch}

Theorem~\ref{thm:maintheorem} follows by setting $\gamma = \frac{1}{\eta^2}$, $\xi = \frac{1}{\eta}$, and calculation.

%% file: kde.tex
\section{Definitions}
Let $f^*$ be an arbitrary symmetric distribution with mean $\mu$ and variance $\sigma^2$, and let $f_r$ be the $r$-smoothed version of $f^*$ with variance $\sigma_r^2 = \sigma^2 + r^2$. Let $s_r$ be the score function of $f_r$, so that $s_r(x) = \frac{f_r'(x)}{f_r(x)}$. Let $\I_r = \E_{x \sim f_r}\left[ s_r(x)^2\right] = -\E_{x \sim f_r}\left[s_r'(x) \right]$ be the Fisher information of $f_r$.

Let $w_r$ be the density function of $\mathcal N(0, r^2)$. We recall the definition of the Kernel Density Estimate (KDE) $\wh f_r$ from $N$ samples $Y_1, \dots Y_N \sim f^*$, from \eqref{eq:kde_def_main_body}.
\begin{equation}
    \label{eq:kde_def} 
    \wh f_r(x) = \frac{1}{N} \sum_{i=1}^N w_r(x - Y_i)
\end{equation}
It has score function $\wh s_r$ with $\wh s_r(x) = \frac{\wh f_r'(x)}{\wh f_r(x)}$. We recall the definition of $\wh s_r^\clip$, the clipped KDE score from $N$ samples with associated failure probability $\delta$, from \eqref{eq:clipped_kde_def_main_body}.
\begin{equation}
    \label{eq:clipped_kde_def}
    \wh s_r^\clip(x) = \sign(\wh s_r(x)) \cdot \min\left(|\wh s_r(x)|, \frac{2}{r} \sqrt{\log \frac{N}{\log \frac{1}{\delta}}} \right)
\end{equation}
We also recall the definition of $\wh s_r^\sym$, the symmetrized clipped KDE score from $N$ samples, symmetrized around $y$, from \eqref{eq:sym_kde_def_main_body}.
\begin{equation}
    \label{eq:sym_kde_def}
    \wh s_r^\sym(x) =  
    \begin{cases}
        \wh s_r^\clip(x) & x \ge y\\
        -\wh s_r^\clip(2y - x) & x < y
    \end{cases}
\end{equation}

\section{Clipped Kernel Density Estimate}
In this section, we will analyze the \emph{clipped Kernel Density Estimate} score function $\wh s_r^\clip$ of a distribution. Our main result in this section is Lemma~\ref{lem:clipped_kde_error}, which says that $\wh s_r^\clip$ is a good approximation to the true score function $s_r$ of the $r$-smoothed distribution in a specific sense.
\subsection{Pointwise guarantees}
In this section, we show that for any individual point $x$, the KDE score $\wh s_r(x)$ approximates $s_r(x)$ well, where $s_r$ is the true score function of our $r$-smoothed distribution $f_r$. We begin by showing that the KDE $\wh f_r$ approximates the true density $f_r$ pointwise.

\begin{lemma}[Pointwise density estimate guarantee]
\label{lem:pointwise_density_estimate}
Let $\wh f_r$ be the kernel density estimate of $f_r$ from $N$ samples $Y_1, \dots, Y_N \sim f^*$, given by $$\wh f_r(x) = \frac{1}{N} \sum_{i=1}^N w_r(x - Y_i)$$
where $w_r$ is the pdf of $\mathcal N(0, r^2)$.
For any fixed $x$, when $N \ge \frac{3 \log \frac{2}{\delta}}{f_r(x) r}$ we have that with probability $1 - \delta$ 
$$\left|\wh f_r(x) - f_r(x) \right| \le \sqrt{ \frac{3 f_r(x) \log \frac{2}{\delta}}{ N r}}$$
This holds even for asymmetric $f^*, f_r$.
\end{lemma}
\begin{proof}
For every $x$, we have $$|w_r(x)| \le \frac{1}{r}$$
So, by multiplicative Chernoff, we have for $0 \le \eps \le 1$ $$\Pr_{Y_i \sim f^*}\left[\left|\frac{1}{N} \sum_{i=1}^N w_r(x - Y_i) - \E_{Y \sim f^*}[w_r(x - Y)]\right| \ge \eps \E_{Y \sim f^*}[w_r(x - Y)]\right] \le 2 \exp\left(- \frac{\eps^2 r N f_r(x)}{3}\right)$$
The claim follows.


\end{proof}

The next Lemma shows that the derivative of the KDE $\wh f_r'$ approximates the true derivative of the density function $f_r'$ pointwise.
\begin{lemma}[Pointwise density derivative estimate guarantee]
\label{lem:pointwise_density_derivative_estimate}
Let $\wh f_r$ be the kernel density estimate of $f_r$ from $N$ samples $Y_1, \dots, Y_N \sim f^*$. For a fixed $x$, letting $N \ge \frac{\log\frac{1}{\delta}}{r f_r(x)}$, we have that with probability $1 - \delta$,
$$\left|\wh f_r'(x) - f_r'(x)\right| \lesssim \sqrt{\frac{f_r(x) \log \frac{1}{\delta}}{N r^3}} $$
This holds even for asymmetric $f^*, f_r$.
\end{lemma}
\begin{proof}

\begin{align*}
&\E_{Y \sim f^*}[w_r'(x - Y_i)^2] = \E_{Y \sim f^*}\left[\left( -w_r(x - Y) \frac{(x - Y)}{r^2}\right)^2\right]\\
&= \frac{1}{r^4} \E_{Y \sim f^*} \left[w_r(x - Y)^2 (x - Y)^2 \right]\\
&\lesssim \frac{1}{r^3} \E_{Y \sim f^*} [w_r(x - Y)] \quad \text{since } w_r(x-Y) (x - Y)^2 \le \frac{r}{\sqrt{2 \pi 
 e}}\\
&= \frac{f_r(x)}{r^3}
\end{align*}

Also, $w_r'(x - Y_i) = - \frac{(x - Y_i)}{\sqrt{2 \pi} r^3} e^{-\frac{(x - Y_i)^2}{2 r^2}}$ is bounded in $[-1/r^2, 1/r^2]$. So, by Bernstein's inequality,
\begin{align*}
    \Pr_{Y_i \sim f^*}\left[\left|\frac{1}{N} \sum_{i=1}^N w_r'(x - Y_i) - \E_{Y \sim f^*}[w_r'(x - Y)] \right| \ge \eps \right] &\le 2 \exp\left(-\Omega\left(\frac{\eps^2}{\frac{f_r(x)}{N r^3} + \frac{\eps}{N r^2} }\right) \right)
\end{align*}
So, 
\begin{align*}
    |\wh f_r'(x) - f_r'(x)| \lesssim \sqrt{\frac{f_r(x) \log \frac{1}{\delta}}{N r^3}} + \frac{\log \frac{1}{\delta}}{N r^2}
\end{align*}
Since $N \ge \frac{\log \frac{1}{\delta}}{f_r(x) r}$, the claim follows.
\end{proof}

Finally, we have the main result of this section, which shows that the KDE score $\wh s_r$ approximates the true score $s_r$ pointwise.
\begin{lemma}[Pointwise Score Estimate Guarantee]
\label{lem:pointwise_score}
Let $\wh f_r$ be the kernel density estimate of $f_r$ from $N$ samples $Y_1, \dots, Y_N \sim f^*$. For fixed $x$, $N \ge \frac{6 \log \frac{4}{\delta}}{r f_r(x)}$, and the KDE score $\wh s_r$ defined in \eqref{eq:kde_def}, given by
$$\wh s_r(x) = \frac{\wh f_r'(x)}{\wh f_r(x)}$$
we have that with probability $1 - \delta$,
$$\left| \wh s_r(x)- s_r(x) \right|\lesssim \sqrt{ \frac{ \log \frac{1}{\delta} \log \frac{1}{r f_r(x)}}{ N r^3 f_r(x)}}$$
This holds even for asymmetric $f^*, f_r$.
\end{lemma}

\begin{proof}
    We have, by Lemmas \ref{lem:pointwise_density_estimate} and \ref{lem:pointwise_density_derivative_estimate}, by a union bound, with probability $1 - \delta$, 
    \[
        \frac{\wh f_r'(x)}{\wh f_r(x)} \le \frac{f_r'(x) +  O\left(\sqrt{\frac{f_r(x) \log \frac{1}{\delta}}{N r^3}} \right)}{f_r(x)\left(1 - \sqrt{ \frac{3\log \frac{4}{\delta}}{f_r(x)r N}}\right)}
    \]
    Since $N \ge \frac{6 \log \frac{4}{\delta}}{r f_r(x)}$ we have that $\sqrt{\frac{3 \log \frac{4}{\delta}}{f_r(x) r N}} \le 1/\sqrt{2}$. So,
    \begin{align*}
    \wh s_r(x) - s_r(x) &\lesssim \sqrt{ \frac{ \log \frac{1}{\delta}}{f_r(x) N r^3}} + s_r(x) \sqrt{\frac{\log \frac{1}{\delta}}{f(x) r N}}\\
    &\lesssim \sqrt{ \frac{ \log \frac{1}{\delta} \log \frac{1}{r f_r(x)}}{ N r^3 f_r(x)}} \quad \text{by Lemma~\ref{lem:generic_density_large_score_small}}
    \end{align*}
    
    Similarly,
    \begin{align*}
    \wh s_r(x) - s_r(x) \gtrsim\sqrt{ \frac{ \log \frac{1}{\delta} \log \frac{1}{r f_r(x)}}{ N r^3 f_r(x)}}
    \end{align*}
\end{proof}

\subsection{Close-by scores are close}
In this section, we show that for small enough $\eps$, $s_r(x+\eps)$ and $s_r(x)$ are close to each other for the score function $s_r$. We begin with the following utility lemma.

\begin{lemma}\label{lem:zcond}
 Let $(X, Y, Z_r)$ be the joint distribution such that $Y \sim f^*$, $Z_r \sim \mathcal N(0, r^2)$ are independent, and $X = Y+Z_r \sim f_r$. For any $x$, and $t > 0$, we have $$\Pr_{Z_r} [\abs{Z_r} > rt | X = x] \leq  \frac{e^{-t^2/2}}{\sqrt{2 \pi} r f_r(x)}$$
This holds even for asymmetric $f^*, f_r$.
\end{lemma}
\begin{proof}
Recall that
  $$f_r(x) = \E_{Y \sim f^*} \left[w_r(x -Y)\right]$$
  where $w_r$ is the pdf of $\mathcal N(0, r^2)$.  So, we have
  \[
    \Pr_{X, Z_r} (\{X = x\} \cap \{\abs{Z_r} > rt\}) = \E_{Y \sim f^*}\left[ w_r(x - Y) 1_{\abs{x - Y} >
      rt}\right] \leq \frac{1}{\sqrt{2\pi}r}e^{-t^2/2}
  \]
  Since $w_r(x - Y) = \frac{1}{\sqrt{2 \pi} r} e^{-\frac{(x - Y)^2}{2 r^2}}$.
  Thus
  \[
    \Pr_{Z_r}[\abs{Z_r} > rt | X = x] = \frac{\Pr_{X, Z_r}\left[\{X =  x\} \cap \{\abs{Z_r} > rt\}\right]}{\Pr_{X \sim f_r}(X = x)} \leq \frac{e^{-t^2/2}}{f_r(x)\sqrt{2\pi}r}.
  \]
\end{proof}

The next lemma shows that $f_r(x+\eps)$ is close to $f_r(x)$ for small $\eps$.
\begin{lemma}
\label{lem:density_close}
Let $\eps, x$ be such that $\frac{8|\eps|}{r} \sqrt{\log \frac{1}{r f_r(x)}} \le 1$. We have,
  \[
    \frac{f_r(x + \eps)}{f_r(x)} \leq 1 + \frac{10|\eps|}{r} \sqrt{\log \frac{1}{f_r(x) r}} 
  \]
This holds even for asymmetric $f^*, f_r$.
\end{lemma}
\begin{proof}
Let $(X, Y, Z_r)$ be the joint distribution such that $Y \sim f^*$, $Z_r \sim \mathcal N(0, r^2)$ are independent, and $X = Y+Z_r \sim f_r$.
    For every $x$, by Lemma~\ref{lem:shifted_score_characterization}, we have
  \[
    \frac{f_r(x+\eps)}{f_r(x)} = \E_{Z_r \mid x} \left[e^{ \frac{2\eps Z_r - \eps^2}{2r^2}}\right] \le \E_{Z_r | x}\left[e^{\frac{\eps Z_r}{r^2}} \right]
  \]
  \allowdisplaybreaks
  Without loss of generality, we assume that $\eps > 0$.
  Now, since $\frac{8 |\eps|}{r} \sqrt{\log \frac{1}{r f_r(x)}} \le  1$,
  \begin{align*}
    \E_{Z_r \mid x}\left[ e^{ \eps Z_r/r^2}\right] &\le \left(1 + \frac{8 \eps}{r}\sqrt{\log\frac{1}{f_r(x) r}} \right) + \int_{1+\frac{8\eps}{r} \sqrt{\log \frac{1}{f_r(x) r}}}^\infty \Pr_{Z_r | x}[e^{ \eps z/r^2} \geq u] du\\
      &= 1 + \frac{8\eps}{r} \sqrt{\log \frac{1}{f_r(x) r}}+  \int_{1 + \frac{8\eps}{r} \sqrt{\log \frac{1}{f_r(x) r}}}^\infty \Pr_{Z_r|x}\left[Z_r \geq \frac{r^2\log u}{ \eps}\right] du\\
      &= 1 + \frac{8 \eps}{r} \sqrt{ \log \frac{1}{r f_r(x)}} + \int_{\log\left(1 + \frac{8 \eps}{r} \sqrt{\log \frac{1}{r f_r(x)}} \right)}^\infty \Pr_{Z_r|x}\left[Z_r \ge r v \right] \frac{\eps}{r} e^{\eps v/r} dv\\
      &\quad \left(\text{Substituting } v = \frac{r \log u}{\eps}\right)\\
      &\le 1 + \frac{8\eps}{r} \sqrt{\log \frac{1}{f_r(x) r}} + \int_{4\sqrt{\log \frac{1}{f_r(x) r}}}^\infty \Pr_{Z_r | x}[Z_r \geq rv] \frac{|\eps|}{r}e^{\eps v/r} dv \\ 
      &\leq 1 + \frac{8\eps}{r} \sqrt{\log \frac{1}{f_r(x) r}}+  \frac{|\eps|}{\sqrt{2 \pi}}\int_{4\sqrt{\log \frac{1}{f_r(x) r}}}^\infty \frac{e^{-v^2/2 + \eps v/r} }{r^2 f_r(x)} dv \quad \text{by Lemma~\ref{lem:zcond}}\\
      &\le 1 + \frac{8\eps}{r} \sqrt{\log \frac{1}{f_r(x) r}} + \frac{|\eps| e^{\frac{2\eps^2}{r^2}}}{r^2 f_r(x)} \int_{4\sqrt{\log \frac{1}{f_r(x) r}}}^\infty \frac{e^{-\frac{\left(v - \frac{2\eps}{r}\right)^2}{2}}}{\sqrt{2 \pi}} dv \\ 
      &\le 1 + \frac{8\eps}{r} \sqrt{\log \frac{1}{f_r(x) r}} + \frac{|\eps| e^{\frac{2\eps^2}{r^2}}}{r^2 f_r(x)} \Pr_{W \sim \mathcal N(0, 1)}\left[W \geq 4\sqrt{\log \frac{1}{f_r(x) r}} - \frac{2 \eps}{r}\right]\\ 
      &\leq 1 + \frac{8\eps}{r} \sqrt{\log \frac{1}{f_r(x) r}} + \frac{11 |\eps|}{10 r^2 f_r(x)} \Pr_{W \sim \mathcal N(0, 1)}\left[W \ge \sqrt{2 \log \frac{1}{r f_r(x)}} \right] \\
      &\quad \text{since $f_r(x) \le \frac{1}{\sqrt{2 \pi} r}$, so that $\frac{2 |\eps|}{r} \le \frac{1}{4 \sqrt{\log \frac{1}{r f_r(x)}}} \le \frac{1}{4 \log \sqrt{2 \pi} } \sqrt{\log \frac{1}{r f_r(x)}}$}\\
        &\quad \text{and since $|\eps|/r \le \frac{1}{8\sqrt{\log \frac{1}{rf_r(x)}}} \le \frac{1}{8 \sqrt{\log \sqrt{ 2\pi}}} $ so that $e^{2 \frac{\eps^2}{r^2}} \le 11/10$}\\
      &\le 1 + \frac{8 \eps}{r} \sqrt{\log \frac{1}{r f_r(x)}} + \frac{11|\eps|}{10r}\\
      &\le 1 + \frac{10 |\eps|}{r}\sqrt{\log \frac{1}{r f_r(x)}} \quad \text{since $\frac{11|\eps|}{10} \le \frac{2|\eps|}{r} \sqrt{\log \sqrt{2 \pi}} \le \frac{2|\eps|}{r} \sqrt{\log \frac{1}{r f_r(x)}}$} 
  \end{align*}
  giving the result.
\end{proof}

The next lemma shows that $f_r'(x+\eps)$ is close to $f_r'(x)$ for small $\eps$.
\begin{lemma}[Close-by density derivatives are close]
\label{lem:density_derivative_add_close}
    Let $\eps, x$ be such that $\frac{20 |\eps|}{r} \sqrt{\log \frac{1}{f_r(x) r}} < 1$. We have $$|f_r'(x+\eps) - f_r'(x)| \lesssim \frac{|\eps|}{r^2} f_r(x) \log \left(\frac{1}{r f_r(x)}\right)$$
This holds even for asymmetric $f^*, f_r$.
\end{lemma}

\begin{proof}
Let $w_r$ be the pdf of $\mathcal N(0, r^2)$. We have that $$w_r''(x) = \frac{-r^2 + x^2}{\sqrt{2 \pi} r^5} e^{-\frac{x^2}{2 r^2}} = \frac{-r^2 + x^2}{r^4} w_r(x) = \frac{-1 + 2 \log \left(\frac{1}{\sqrt{2 \pi} r \cdot w_r(x)}\right)}{r^2} w_r(x)$$
since $x^2 = 2 r^2 \log \left(\frac{1}{\sqrt{2 \pi} r \cdot w_r(x)}\right)$. So, since $g(z) = z \log \left(\frac{1}{\sqrt{2 \pi} r \cdot z} \right)$ is concave on $[0, 1]$, we have
\begin{align*}
&\E_{Y \sim f^*}[w_r''(x - Y)]\\
&= \frac{1}{r^2} \E_{Y \sim f^*} \left[w_r(x - Y)\left(-1 + 2 \log \left(\frac{1}{\sqrt{2 \pi} r \cdot w_r(x - Y)} \right)\right)\right]\\
&\le \frac{1}{r^2}\E_{Y \sim f^*}[w_r(x - Y)]\left(-1 + 2 \log \left(\frac{1}{\sqrt{2 \pi} r \E_{Y \sim f^*}[w_r(x - Y)]} \right)  \right) \quad \text{by Jensen's inequality}\\
&\lesssim \frac{f_r(x) \log \left(\frac{1}{f_r(x)\sqrt{2 \pi} r}\right)}{r^2}
\end{align*}
So, by Taylor's theorem, for some $|\zeta| < |\eps|$
\begin{align*}
w_r'(x + \eps - Y) = w_r'(x - Y) + \eps w_r''(x + \zeta - Y)
\end{align*}
Now, by the above
\begin{align*}
    \E_{Y \sim f^*}\left[w''_r(x + \zeta - Y)\right] &\lesssim \frac{f_r(x+\zeta) \log \left(\frac{1}{f_r(x+\zeta) \sqrt{2 \pi} r}\right)}{r^2}\\
    &\lesssim \frac{f_r(x) \log\left(O\left(\frac{1}{f_r(x) r}\right) \right)}{r^2}\quad \text{Using Lemma~\ref{lem:density_close}}\\ 
    &\lesssim \frac{f_r(x) \log \frac{1}{r f_r(x)}}{r^2} \quad \text{since $r f_r(x) \le \frac{1}{\sqrt{2 \pi}}$}
\end{align*}
So, by the above, 
\begin{align*}
    f_r'(x + \eps) - f_r'(x) &\lesssim  \eps\E_{Y \sim f^*}\left[ w_r''(x + \zeta - Y) \right]\\
    &\lesssim  \frac{\eps}{r^2} f_r(x) \log \left(\frac{1}{r f_r(x)} \right)
\end{align*}
Similarly, by Taylor's theorem, for some $|\zeta| < |\eps|$,
$$w_r'(x - Y) = w_r'(x + \eps - Y) - \eps w_r''(x + \zeta - Y)$$
so that
\begin{align*}
    f_r'(x) - f_r'(x + \eps) \lesssim -\frac{\eps}{r^2} f_r(x) \log \left(\frac{1}{r f_r(x)} \right)
\end{align*}
The claim follows.
\end{proof}

Finally, the main result of this section shows that $s_r(x+\eps)$ is close to $s_r(x)$ for small $\eps$.
\begin{lemma}[Close-by scores are close]
\label{lem:close_by_scores_close}
    Let $\eps, x$ be such that $\frac{20|\eps|}{r} \sqrt{\log \frac{1}{r f_r(x)}} < 1$. We have $$\left|s_r(x+\eps) - s_r(x) \right| \lesssim \frac{|\eps|}{ r^{2}} \log \left(\frac{1}{r f_r(x)}\right)$$
    This holds even for asymmetric $f^*, f_r$.
\end{lemma}
\begin{proof}
  By Lemma~\ref{lem:density_close}~and~\ref{lem:density_derivative_add_close}, we have $$ \frac{f'_r(x+\eps)}{f_r(x+\eps)} \le \frac{f_r'(x) + O\left(\frac{|\eps|}{r^2} f_r(x) \log\left( \frac{1}{r f_r(x)}\right)\right)}{f_r(x)\left(1 - \frac{10|\eps|}{r} \sqrt{ \log \frac{1}{r f_r(x)}}\right) }$$
  Since $\frac{10 |\eps|}{r} \sqrt{\log \frac{1}{r \alpha}} < 1/2$, we have 
  \begin{align*}
  \frac{f_r'(x+\eps)}{f_r(x+\eps)} \le \frac{f_r'(x)}{f_r(x)} + O\left(\frac{|\eps|}{r^2}\log \left(\frac{1}{r f_r(x)} \right)\right) + O\left(\frac{f_r'(x)}{f_r(x)} \frac{|\eps|}{r} \sqrt{\log \frac{1}{rf_r(x)}}\right)
  \end{align*}
  So, 
  \begin{align*}
  s_r(x+\eps) - s_r(x) &\lesssim + \frac{|\eps|}{r^{2}}\log \left(\frac{1}{r f_r(x)}\right) + s_r(x)  \frac{|\eps|}{r}\sqrt{ \log \frac{1}{rf_r(x)}}\\
  &\lesssim \frac{|\eps|}{r^2} \log \left(\frac{1}{r f_r(x)}\right) \quad \text{by Lemma~\ref{lem:generic_density_large_score_small}}
  \end{align*}
 
    We can get the lower bound in the same way.
\end{proof}
\subsection{Bounding the clipped KDE error}
In this section, we show that the clipped KDE score function $\wh s_r^\clip$ approximates the true score function $\wh s_r^\sym$ in a specific sense. We begin by showing that sets that have small density under $f_r$ have small expected score.
\begin{lemma}
    \label{lem:score_bounded_in_small_set}
  Let $S$ be any set with $\Pr_{f_r}[S] \leq \beta$. Let $|\eps| \le r/2$ and $|\eps| \le \frac{r}{\sqrt{\log \frac{1}{r^2 \I_r}}}$. Then, for the score function $s_r$ of $f_r$, we have
  \[
  \E_{x \sim f_r}[s_r(x+\eps)^2 1_{x \in S}] \lesssim \frac{\beta}{r^2}\log \frac{1}{\beta}
  \]
    This holds even for asymmetric $f^*, f_r$.
\end{lemma}
\begin{proof}
  \begin{align*}
      \E_{x \sim f_r} [s_r^2(x+\eps) \1_{x \in S}]
      &= \int_0^\infty \Pr_{x \sim f_r}[s_r^2(x+\eps) 
      \1_{x \in S} \ge t] dt\\
      &\leq \int_0^\infty \min(\Pr_{x \sim f_r}[x \in S], \Pr_{x \sim f_r}[s_r^2(x+\eps) \ge t]) dt\\
      &\leq \int_0^\infty \min(\beta, \Pr_{x \sim f_r}[|s_r(x+\eps)| \ge \sqrt{t}]) dt\\
    \end{align*}
    So, by Lemma~\ref{lem:score-subgaussian}, for some explicit constant $C > 0$, we have:
    \begin{align*}
      \E_{x \sim f_r} [s_r^2(x+\eps) \1_{x \in S}]&\lesssim \int_0^\infty \min(\beta, e^{-C t r^2} )dt\\
      &\lesssim \beta B + \int_B^\infty e^{-C t r^2} dt\\
      &\lesssim \beta B + \frac{ e^{-CBr^2}}{Cr^2}\\
  \end{align*}
  Thus, setting $B = \frac{\log \frac{1}{\beta}}{Cr^2}$ gives 
  \begin{align*}
      \E[s_r^2(x+\eps) 1_{x \in S}]\lesssim \frac{\beta}{r^2} \log \frac{1}{\beta}
  \end{align*}
 \end{proof}
 The next lemma shows that for every small width region, the KDE score $\wh s$ approximates the true score well, as long as the density of that region is large.
\begin{lemma}[Generic Error estimate within bin]
\label{lem:generic_error_estimate_within_bin}
Let $f^*$ be an arbitrary distribution and let $f_r$ be the $r$-smoothed version of $f^*$. Let $\wh f_r$ be the kernel density estimate of $f_r$ from $N$ samples $Y_1, \dots, Y_N \sim f^*$. Let $x, \eps, N$ be such that $N \ge \frac{6 \log \frac{4}{\delta}}{ r f_r(x)}$, and $\frac{20|\eps|}{r} \sqrt{\log \frac{1}{r f_r(x) }} < 1$. Then, for the KDE score $\wh s_r$ defined in \eqref{eq:kde_def}, with probability $1 - \delta$, we have that for all $|\zeta| \le |\eps|$ (simultaneously),
$$|\wh s_r(x + \zeta) - s_r(x + \zeta)| \lesssim\sqrt{ \frac{ \log \frac{1}{\delta} \log \frac{1}{r f_r(x)}}{ N r^3 f_r(x)}} + \frac{|\eps| }{r^2} \log \left(\frac{1}{r f_r(x)} \right)$$
\end{lemma}
\begin{proof}
First, by Lemma \ref{lem:pointwise_score}, we have that with probability $1 - \delta$, $$|\wh s_r(x) - s_r(x)| \lesssim \sqrt{\frac{\log \frac{1}{\delta} \log \frac{1}{r f_r(x)}}{N r^3 f_r(x)}}$$
Now, by Lemma~\ref{lem:close_by_scores_close}, since $|\zeta| \le |\eps|$,
$$|s_r(x+\zeta) - s_r(x)| \lesssim \frac{|\eps|}{r^2} \log\left( \frac{1}{r f_r(x)}\right)$$
and
\begin{align*}
|\wh s_r(x+\zeta) - \wh s_r(x)| &\lesssim \frac{|\eps|}{r^2} \log\left(\frac{1}{r \wh f_r(x)} \right)\\
&= \frac{|\eps|}{r^2} \left(\log\left(\frac{1}{r f_r(x)}\right) + \log\left(\frac{f_r(x)}{\wh f_r(x)}\right)\right)\\
&\le \frac{|\eps|}{r^2} \left(\log\left(\frac{1}{r f_r(x)}\right) + \log\left(1+\frac{1}{\sqrt{2}}\right)\right) \quad \text{by Lemma~\ref{lem:pointwise_density_estimate} and since $N \ge \frac{6 \log \frac{4}{\delta}}{r f_r(x)}$}\\
&\lesssim \frac{|\eps|}{r^2} \log \left(\frac{1}{r f_r(x)} \right) \quad \text{since $r f_r(x) \le \frac{1}{\sqrt{2\pi}}$, so $\log\frac{1}{r f_r(x)} = \Omega(1)$}
\end{align*}
Putting everything together, with probability $1 - \delta$, for all $|\zeta| \le |\eps|$,
\begin{align*}
    |\wh s_r(x+\zeta) - s_r(x+\zeta)| &\le |\wh s_r(x+\zeta) - \wh s_r(x)| + |\wh s_r(x) - s_r(x)| + |s_r(x) - s_r(x+\zeta)|\\
    &\lesssim \sqrt{\frac{\log \frac{1}{\delta} \log \frac{1}{r f_r(x)}}{N r^3 f_r(x)}} + \frac{|\eps|}{r^2} \log \left(\frac{1}{r f_r(x)} \right)
\end{align*}
\end{proof}

%
The next lemma shows that for all points with density larger than $\alpha$ within a $t\sigma$ radius around the true mean, the KDE score approximates the true score well.
\begin{lemma}[Generic Error estimate over large density region]
\label{lem:generic_error_estimate_large_density}
Let $f^*$ be an arbitrary distribution with mean $\mu$ and variance $\sigma^2$, and let $f_r$ be the $r$-smoothed version of $f^*$, with variance $\sigma_r^2 = \sigma^2 + r^2$. Let $\wh s_r$ be the score of the kernel density estimate of $f_r$ from $N$ samples $Y_1, \dots, Y_N \sim f^*$, as defined in \eqref{eq:kde_def}.
Let $\alpha > 0$ and let $N \ge \frac{6\log \left(\frac{4}{\delta} \left(\frac{2\sqrt{\alpha N}t \sigma_r}{\sqrt{r} } +1 \right)\right) + 400 \log \frac{1}{\alpha r}}{\alpha r}$. Then with probability $1 - \delta$, we have that,
\begin{align*}
    \E_{x \sim f_r}\left[(\wh s_r(x) - s_r(x))^2 \1_{\{|x - \mu| \le t \sigma_r \text{ and } f_r(x) \ge \alpha\}}\right] \lesssim \frac{\log\left(\frac{1}{\delta} \left(\frac{2t \sigma_r\sqrt{\alpha N} }{\sqrt{r} } + 1\right)\right)}{\alpha N r^3} \log^2 \frac{1}{\alpha r}
\end{align*}

\end{lemma}

\begin{proof}
Consider contiguous intervals of length $\eps$ starting from $\mu - t\sigma_r$ so that the last interval covers $\mu + t\sigma_r$, and let $S$ be the set of the smallest $y$ such that $f(y)\ge \alpha$ in each of these intervals, if one exists. Note that $|S| \le \frac{2 t \sigma_r}{\eps} + 1$. Then, we have that $$\left\{x : |x - \mu| \le t \sigma_r \text{ and } f_r(x) \ge \alpha \right\} \subseteq \left\{[y - \eps, y + \eps] | y \in S \right\}$$
   Now, for $\eps = \sqrt{\frac{r}{\alpha N}}$ and $y \in S$, since $N \ge \frac{6\log \left(\frac{4}{\delta} \left(\frac{2\sqrt{\alpha N}t \sigma_r}{\sqrt{r} } +1 \right)\right)}{\alpha r} \ge \frac{6 \log \frac{4 |S|}{\delta}}{ r f_r(y)}$ and $\frac{20 |\eps|}{r} \sqrt{ \log \frac{1}{r f_r(y)}} \le \frac{20}{\sqrt{ \alpha r N}}\sqrt{\log \frac{1}{\alpha r}} \le 1$, by Lemma~\ref{lem:generic_error_estimate_within_bin}, we have that with probability $1 - \frac{\delta}{|S|}$, for all $|\zeta| \le \eps$ (simultaneously),
   \begin{align*}
        |\wh s_r(y + \zeta) - s_r(y + \zeta)| &\lesssim \sqrt{\frac{\log \frac{|S|}{\delta} \log \frac{1}{r f_r(y)}}{N r^3 f_r(y)}} + \sqrt{\frac{1}{\alpha N r^3}} \log\frac{1}{r f_r(y)}\\
        &\lesssim \sqrt{\frac{\log\left(\frac{1}{\delta} \left(\frac{2t \sigma_r\sqrt{\alpha N} }{\sqrt{r} } + 1\right)\right)}{\alpha N r^3}} \log \frac{1}{\alpha r}
    \end{align*}
   So by a union bound, with probability $1 - \delta$, for all $x$ such that $|x - \mu| \le t \sigma_r$ and $f(x) \ge \alpha$ simultaneously,
   $$|\wh s_r(x) - s_r(x)| \lesssim \sqrt{\frac{\log\left(\frac{1}{\delta} \left(\frac{2t \sigma_r\sqrt{\alpha N} }{\sqrt{r} } + 1\right)\right)}{\alpha N r^3}} \log \frac{1}{\alpha r}$$
   So, $$\E_{x \sim f_r}[(\wh s_r(x) - s_r(x))^2 \1_{|x - \mu| \le t \sigma_r \text{ and } f_r(x) \ge \alpha}] \lesssim \frac{\log\left(\frac{1}{\delta} \left(\frac{2t \sigma_r\sqrt{\alpha N} }{\sqrt{r} } + 1\right)\right)}{\alpha N r^3} \log^2 \frac{1}{\alpha r}$$
\end{proof}
The next lemma instantiates the previous one with a particular value of $t$ and $\alpha$ based on our desired failure probability and the number of samples.
\begin{lemma}[Error estimate over large density region (instantiated)]
\label{lem:error_estimate_large_density_instantiated}
   Let $f^*$ be an arbitrary distribution with mean $\mu$ and variance $\sigma^2$, and let $f_r$ be the $r$-smoothed version of $f^*$, with variance $\sigma_r^2 = \sigma^2 + r^2$ and Fisher information $\I_r$. Let $\wh s_r$ be the score of the kernel density estimate of $f_r$ from $N$ samples $Y_1, \dots, Y_N \sim f^*$, as defined in \eqref{eq:clipped_kde_def}. Let $\gamma \ge C$ for large enough constant $C > 1$ be a parameter. Let $t = \sqrt{\frac{\gamma \log \frac{N}{\I_r r^2\log \frac{1}{\delta}}}{\I_r r^2}}$, $\alpha = \frac{1}{t^3\sigma_r}$. Then for any $r \le \sigma$ and $\frac{N}{\log \frac{1}{\delta} } \ge \left(\gamma^{5/12}\frac{\sigma }{r} \right)^{6 + \beta}$ for any constant $\beta > 0$, with probability $1 - \delta$, we have that,
    \begin{align*}
    \E_{x \sim f_r}\left[(\wh s_r(x) - s_r(x))^2 \1_{\{|x - \mu| \le t \sigma_r \text{ and } f_r(x) \ge \alpha\}}\right] \lesssim \frac{\I_r}{\gamma}
\end{align*}
\end{lemma}
\begin{proof}
    First, note that since $r\le \sigma$,
    \begin{align*}
        \sigma_r^2 = \sigma^2 + r^2 \le 2 \sigma^2
    \end{align*}
    Also, our setting of $N$ implies that WLOG
    \[
        \frac{N}{\log \frac{1}{\delta}} \ge \left( \frac{\gamma^{5/12} \sigma \log \frac{N}{\log \frac{1}{\delta}}}{r}\right)^{6}
    \]
    since $N \ge C \log \frac{1}{\delta}$.
    So,
    \begin{equation}
    \label{eq:sigma_over_r_bound}
        \frac{\sigma}{r} \le\left(\frac{N}{\log \frac{1}{\delta}} \right)^{1/6}\cdot \frac{1}{\gamma^{5/12} \log \frac{N}{\log \frac{1}{\delta}}} \le \left(\frac{N}{\log\frac{1}{\delta}}\right)^{1/6}
    \end{equation}
    We will first check that this $N$ satisfies the condition required to invoke Lemma~\ref{lem:generic_error_estimate_large_density} that $N \ge \frac{6\log \left(\frac{4}{\delta} \left(\frac{2\sqrt{\alpha N}t \sigma_r}{\sqrt{r} } +1 \right)\right) + 400 \log \frac{1}{\alpha r}}{\alpha r}$. To do this, we will individually upper bound $\frac{1}{\alpha r}$ and $\frac{2 t \sigma_r \sqrt{\alpha N}}{\sqrt{r}}$.

    We have,
    \begin{align*}
        \frac{1}{\alpha r} &= \frac{\sigma_r}{r} \left(\frac{\gamma \log \frac{N}{\I_r r^2\log \frac{1}{\delta}}}{\I_r r^2} \right)^{3/2} \quad \text{since $\alpha = \frac{1}{t^3 \sigma_r}$ and $t =  \sqrt{\frac{\gamma \log \frac{N}{\I_r r^2\log \frac{1}{\delta}}}{\I_r r^2}}$}\\
        &\le \left(\frac{\sigma_r}{r}\right)^4 \gamma^{3/2} \log^{3/2} \frac{N \sigma_r^2}{r^2 \log \frac{1}{\delta}} \quad \text{since $\I_r \ge \frac{1}{\sigma_r^2}$}\\
        &\le \left(\frac{2\sigma}{r}\right)^4 \gamma^{3/2} \log^{3/2} \left(\frac{2 N \sigma^2}{r^2\log \frac{1}{\delta}}\right)\quad \text{since $\sigma_r^2 \le 2 \sigma^2$}\\
        &\le 16 \frac{N^{4/6}}{\gamma^{5/3} \log^4(\frac{N}{\log \frac{1}{\delta}}) \log^{4/6} \frac{1}{\delta}} \gamma^{3/2} \log^{3/2} \left(2\left(\frac{N}{\log \frac{1}{\delta}}\right)^{4/3}\right) \quad \text{by \eqref{eq:sigma_over_r_bound}}\\
        &\le \frac{N}{\gamma^{1/10}\log (\frac{1}{\delta}) \log^2 \left(\frac{N}{\log \frac{1}{\delta}}\right)} \quad \text{since $\gamma \ge C$ for $C$ large enough constant, and $\frac{\log \frac{1}{\delta}}{N} \le 1$} 
    \end{align*}
    To further justify the last line above, observe that $\frac{\gamma^{3/2}}{\gamma^{5/3}} \le \frac{1}{\gamma^{1/6}} = \frac{1}{\gamma^{1/10} \cdot \gamma^{1/15}}$, and that for large enough constant $C$, since $\gamma > C$, $\gamma^{1/15}$ can be made larger than any fixed constant. Also note that $\log^{3/2}(2 (\frac{N}{\log \frac{1}{\delta}})^{4/3}) \le \log^{2} \frac{N}{\log \frac{1}{\delta}}$ for large enough $C$ since $\frac{N}{\log \frac{1}{\delta}} \ge \gamma\ge C$. So, the inequality follows.

    Next, we bound $\frac{2 t \sigma_r \sqrt{\alpha N}}{\sqrt{r}}$.
\begin{align*}
    \frac{2 t \sigma_r \sqrt{\alpha N}}{\sqrt{r}} &= 2\sqrt{ \frac{N \sigma_r}{t r}} \quad \text{since $\alpha = \frac{1}{t^3 \sigma_r}$}\\
    &= 2 \sqrt{\frac{N \sigma_r}{\sqrt{\frac{\gamma \log \frac{N}{\I_r r^2\log\frac{1}{\delta}}}{\I_r r^2}} r}} \quad \text{since $t =  \sqrt{\frac{\gamma \log \frac{N}{\I_r r^2\log \frac{1}{\delta}}}{\I_r r^2}}$}\\
    &\le 4 \sqrt{\frac{N \sigma}{r \sqrt{ \gamma \log \frac{N}{\log \frac{1}{\delta}}} }} \quad \text{since $\I_r \le \frac{1}{r^2}$ and $\sigma_r^2 \le 2 \sigma^2$}\\
    &\le 4 \sqrt{\frac{N}{\sqrt{\gamma \log \frac{N}{\log \frac{1}{\delta}}}} \cdot \left(\frac{N}{\log \frac{1}{\delta}}\right)^{1/6}} \quad \text{by \eqref{eq:sigma_over_r_bound}}\\
    &\le 4 N \cdot \left(\frac{N}{\log \frac{1}{\delta}} \right)^{1/12} \quad \text{since $\gamma \ge 1$, $\frac{N}{\log \frac{1}{\delta}} \ge 1$}
\end{align*}
So, we can now check the condition required to invoke Lemma~\ref{lem:generic_error_estimate_large_density}. We have,
\begin{align*}
    &\frac{6 \log\left( \frac{4}{\delta} \left(\frac{2 \sqrt{\alpha N} t \sigma_r}{\sqrt{r}} + 1 \right)\right) + 400 \log \frac{1}{\alpha r}}{\alpha r}\\
    &\le \left(6 \log \left(\frac{4}{\delta}\left(4 N \cdot \left(\frac{N}{\log \frac{1}{\delta}} \right)^{1/12} + 1 \right) \right) + 400 \log \frac{N}{\log \frac{1}{\delta}} \right)\left(\frac{N}{\gamma^{1/10} \log^2 \left( \frac{N}{\log \frac{1}{\delta}}\right)\log\frac{1}{\delta}} \right)\\
    &\le N \quad \text{since $\gamma \ge C$}
\end{align*}
So, by Lemma~\ref{lem:generic_error_estimate_large_density}, we have
\begin{align*}
     \E_{x \sim f_r}\left[(\wh s_r(x) - s_r(x))^2 \1_{\{ |x - \mu| \le t \sigma_r \text{ and } f_r(x) \ge \alpha \}} \right] &\lesssim \frac{\log \left( \frac{1}{\delta} \left(\frac{2 \sigma_r \sqrt{\alpha N}}{\sqrt{r}} + 1 \right)\right)}{\alpha N r^3} \log^2 \frac{1}{\alpha r}\\
\end{align*}
To bound the RHS above by $O\left(\I_r/\gamma\right)$ as required, we will first bound $\frac{1}{\alpha N r^3}$.
\begin{align*}
    \frac{1}{\alpha N r^3} &= \frac{\sigma_r}{N r^3} \left(\frac{\gamma \log \frac{N}{\I_r r^2\log \frac{1}{\delta}}}{\I_r r^2} \right)^{3/2}\quad \text{since $\alpha = \frac{1}{t^3 \sigma_r}$ and $t =  \sqrt{\frac{\gamma \log \frac{N}{\I_r r^2\log \frac{1}{\delta}}}{\I_r r^2}}$}\\
    &= \frac{\I_r}{\gamma} \left(\frac{\sigma_r}{N r^6} \frac{\gamma^{5/2}}{\I_r^{5/2}} \log^{3/2} \frac{N}{ \I_r r^2\log \frac{1}{\delta}}\right)\\
    &\le\frac{\I_r}{\gamma} \left( \frac{\sigma_r^6\gamma^{5/2} \log^{3/2} \frac{N \sigma_r^2}{r^2 \log \frac{1}{\delta}}}{N r^6}\right) \quad \text{since $\I_r \ge \frac{1}{\sigma_r^2}$}\\
    &\lesssim \frac{\I_r}{\gamma} \left( \frac{\sigma^6\gamma^{5/2} \log^{3/2} \frac{4N \sigma^2}{r^2 \log \frac{1}{\delta}}}{N r^6}\right) \quad \text{since $\sigma_r^2 \le 2 \sigma^2$}\\
    &\le \frac{\I_r}{\gamma} \left(\frac{\log^{3/2} (4 \left(\frac{N}{\log \frac{1}{\delta}}\right)^{4/3} )}{\log^6 \left(\frac{N}{\log \frac{1}{\delta}}\right) \log \frac{1}{\delta} } \right) \quad \text{by \eqref{eq:sigma_over_r_bound}}\\
    &\lesssim \frac{\I_r}{\gamma \log^4 \left(\frac{N}{\log \frac{1}{\delta}} \right)\log \frac{1}{\delta}}
\end{align*}
So, plugging in everything,
\begin{align*}
    &\E_{x \sim f_r}\left[(\wh s_r(x) - s_r(x))^2 \1_{\{ |x - \mu| \le t \sigma_r \text{ and } f_r(x) \ge \alpha \}} \right]\\
    &\lesssim \frac{\log \left( \frac{1}{\delta} \left(\frac{2 \sigma_r \sqrt{\alpha N}}{\sqrt{r}} + 1 \right)\right)}{\alpha N r^3} \log^2 \frac{1}{\alpha r}\\
    &\lesssim \frac{\I_r}{\gamma \log^4 \left(\frac{N}{\log \frac{1}{\delta}}\right) \log\frac{1}{\delta}} \log\left(\frac{4N \cdot \left(\frac{N}{ \log \frac{1}{\delta}} \right)^{1/12}+ 1}{\delta} \right) \log^2 \left(\left(\frac{N}{\log \frac{1}{\delta}}\right)^{4/6}\right)\\
    &\lesssim \frac{\I_r}{\gamma}
\end{align*}
\end{proof}
The lemmas so far have shown that the KDE score $\wh s_r$ approximates the true score $s_r$ well in large denstiy regions in the typical $t \sigma$ radius around the true mean. The next lemma shows that the same guarantee holds for the clipped KDE score $\wh s_r^\clip$.
\begin{lemma}[Error estimate over large density region for clipped KDE]
\label{lem:error_estimate_large_density_clipped_kde}
   Let $f^*$ be an arbitrary distribution with mean $\mu$ and variance $\sigma^2$, and let $f_r$ be the $r$-smoothed version of $f^*$, with variance $\sigma_r^2 = \sigma^2 + r^2$. Let $\wh s_r^\clip$ be the clipped kernel density estimate score from $N$ samples, as defined in \eqref{eq:kde_def}. Let $\gamma \ge C$ for large enough constant $C > 1$ be a parameter. Let $t = \sqrt{\frac{\gamma \log \frac{N}{\I_r r^2\log \frac{1}{\delta}}}{\I_r r^2}}$, $\alpha = \frac{1}{t^3\sigma_r}$.Then for any $r \le \sigma$ and $\frac{N}{\log \frac{1}{\delta}} \ge \left(\frac{\gamma^{5/12} \sigma}{r} \right)^{6 + \beta}$ for $\beta > 0$, with probability $1 - \delta$, we have that,
    \begin{align*}
    \E_{x \sim f_r}\left[(\wh s_r^\clip(x) - s_r(x))^2 \1_{\{|x - \mu| \le t \sigma_r \text{ and } f_r(x) \ge \alpha\}}\right] \lesssim \frac{\I_r}{\gamma}
\end{align*}
\end{lemma}
\begin{proof}
    Note that our condition that $\frac{N}{\log \frac{1}{\delta}} \ge \left(\frac{\gamma^{5/12} \sigma }{r} \right)^{6 + \beta}$ implies
    \[
        \frac{\sigma}{r} \le \left(\frac{N}{\log \frac{1}{\delta}}\right)^{1/6}
    \]
    Also, since $r \le \sigma$, $\frac{N}{\log \frac{1}{\delta}} \ge (\gamma^{5/12})^{6 + \beta} \ge \gamma^{5/2}$.
    
    So, by Lemma~\ref{lem:generic_density_large_score_small}, for $x$ such that $f_r(x) \ge \alpha$,
    \begin{align*}
        |s_r(x)| &\le \frac{1}{r} \sqrt{2 \log \frac{1}{\sqrt{2 \pi} r \alpha}}\\
        &= \frac{1}{r} \sqrt{2 \log \left(\frac{\sigma_r}{\sqrt{2 \pi} r} \left(\frac{\gamma \log \frac{N}{\I_r r^2 \log \frac{1}{\delta}}}{\I_r r^2} \right)^{3/2} \right)} \quad \text{since $\alpha = \frac{1}{t^3 \sigma_r}$ and $t =  \sqrt{\frac{\gamma \log \frac{N}{\I_r r^2\log \frac{1}{\delta}}}{\I_r r^2}}$}\\
        &\le \frac{1}{r} \sqrt{2 \log \left(\frac{\sigma_r^4}{\sqrt{2 \pi} r^4} \left(\gamma \log \frac{\sigma_r^2 N}{r^2\log \frac{1}{\delta}}\right)^{3/2} \right)} \quad \text{since $\I_r \ge \frac{1}{\sigma_r^2}$}\\
        &\le \frac{1}{r} \sqrt{2 \log\left(\frac{16}{\sqrt{2 \pi}} \left(\frac{N}{\log \frac{1}{\delta}} \right)^{4/6} \left(\gamma \log \left(4\left(\frac{N}{\log \frac{1}{\delta}}\right)^{4/3}\right)\right)^{3/2}\right)}\\
        &\quad \text{since $\sigma_r^2 \le 2\sigma^2$ and using $\sigma/r \le \left(\frac{N}{\log \frac{1}{\delta}}\right)^{1/6}$}\\
        &\le \frac{2}{r} \sqrt{ \log \frac{N}{\log \frac{1}{\delta}}} \quad \text{since $\frac{N}{\log \frac{1}{\delta}} \ge \gamma^{5/2} \ge C^{5/2}$ for a sufficiently large constant $C$}
    \end{align*}
    Then, since $\wh s_r^\clip$ is clipped at $\frac{2}{r} \sqrt{\log \frac{N}{\log \frac{1}{\delta}}}$ by definition in \eqref{eq:clipped_kde_def}, by Lemma~\ref{lem:error_estimate_large_density_instantiated}, we have the claim.
\end{proof}

The next lemma shows that for small density sets, for any function that is clipped appropriately, the error incurred relative to the true score function is small.
\begin{lemma}
\label{lemma:score_error_bounded_in_small_set}
Let $f^*$ be an arbitrary distribution, and let $f_r$ be the $r$-smoothed version of $f^*$. Let $s_r$ be the score function of $f_r$, and let $\I_r$ be the Fisher information. Let $\gamma \ge C$ for large enough constant $C$, $\frac{N}{\log \frac{1}{\delta}} \ge \gamma$, and let $\Tilde s$ be any function with $|\Tilde s_r(x)| \le \frac{2}{r} \sqrt{\log \frac{N}{\log \frac{1}{\delta}}}$ for all $x$. Let $S$ be a set with $\Pr[S] \le \frac{1}{t^2}$ for $t = \sqrt{\frac{\gamma \log \frac{N}{\I_r r^2\log \frac{1}{\delta}}}{\I_r r^2} }$. Then, we have $$\E_{x \sim f_r}[(\Tilde s_r(x) - s_r(x))^2 \1_{x \in S}] \lesssim \frac{\I_r}{\gamma} $$

\end{lemma}
\begin{proof}
First, by Lemma~\ref{lem:score_bounded_in_small_set},
\begin{align*}
    \E_{x \sim f_r}\left[s_r(x)^2 \1_{x \in S} \right] &\lesssim \frac{1}{t^2 r^2} \log t^2\\
    & \le \frac{\I_r}{\gamma \log \frac{N}{\I_r r^2 \log \frac{1}{\delta}}} \log \left(\frac{\gamma}{\I_r r^2} \log \frac{N}{\I_r r^2\log \frac{1}{\delta}} \right)\\
    &\lesssim \frac{\I_r}{\gamma} \quad \text{since $\gamma \ge 1$, $\gamma \le \frac{N}{\log \frac{1}{\delta}}$ and $\I_r \le \frac{1}{r^2}$}
\end{align*}
Now, by assumption, 
$$|\Tilde s_r(x)| \le \frac{2}{r} \sqrt{\log \frac{N}{\log \frac{1}{\delta}}}$$
So,
\begin{align*}
    \E_{x \sim f_r}\left[ \Tilde s_r(x)^2 \1_{x \in S} \right] &\le \frac{4}{t^2 r^2} \log \frac{N}{\log \frac{1}{\delta}} \\
    &= \frac{\I_r}{\gamma \log \frac{N}{\I_r r^2\log \frac{1}{\delta}}} \log\left( \frac{N}{\log \frac{1}{\delta}}  \right)\\
    &\lesssim \frac{\I_r}{\gamma \log \frac{N}{\I_r r^2\log \frac{1}{\delta}}} \log\left( \frac{N}{\I_r r^2\log \frac{1}{\delta}}  \right) \quad \text{since $\I_r \le \frac{1}{r^2}$}\\
    &\lesssim \frac{\I_r}{\gamma}
\end{align*}
Thus, we have
\begin{align*}
    \E_{x \sim f_r}\left[(\Tilde s_r(x) - s_r(x))^2 \1_{x \in S} \right] &\lesssim \E_{x \sim f_r}\left[\Tilde s_r(x)^2 \1_{x \in S} \right] + \E_{x \sim f_r}\left[s_r(x)^2 \1_{x \in S} \right]\\
    &\lesssim \frac{\I_r}{\gamma}
\end{align*}
\end{proof}
The next lemma shows that within the typical $t \sigma$ radius around the mean, for the set of points with small density, the clipped KDE score $\wh s_r^\clip$ approximates the true score $s_r$ well.
\begin{lemma}[Error estimate over small density regions within typical region]
\label{lem:clipped_kde_small_density_typical}
Let $f^*$ be an arbitrary distribution with mean $\mu$ and variance $\sigma^2$, and let $f_r$ be the $r$-smoothed version of $f^*$, with variance $\sigma_r^2 = \sigma^2 + r^2$. Let $s_r$ be the score function of $f_r$, and let $\I_r$ be the Fisher information. Let $\Tilde s_r$ be any score function with $|\Tilde s_r(x)| \le \frac{2}{r} \sqrt{\log \frac{N}{\log \frac{1}{\delta}}}$ for all $x$. Let $t =\sqrt{\frac{\gamma \log \frac{N}{\I_r r^2 \log \frac{1}{\delta}}}{\I_r r^2} } $ for $\gamma \ge 1$, and let $\alpha = \frac{1}{t^3 \sigma_r}$. Then,
    \begin{align*}
        \E_{x \sim f_r}[(\Tilde s_r(x) - s_r(x))^2 \1_{\{|x - \mu| \le t \sigma_r \text{ and } f(x) < \alpha\}}] \lesssim \frac{\I_r}{\gamma}
    \end{align*}
\end{lemma}
\begin{proof}
By our choice of $\alpha$,
    $$\Pr_{x \sim f_r}[|x - \mu| \le t \sigma_r \text{ and } f(x) < \alpha] \le \alpha t \sigma_r = \frac{1}{t^2}$$
    So, by Lemma~\ref{lemma:score_error_bounded_in_small_set}, the claim follows
\end{proof}

The next lemma shows that in the region outside the typical region of radius $t \sigma$ around the true mean, the clipped KDE score $\wh s_r^\clip$ approximates the true score well.
\begin{lemma}[Error estimate over atypical region]
\label{lem:clipped_kde_atypical}
Let $f^*$ be an arbitrary distribution with mean $\mu$ and variance $\sigma^2$, and let $f_r$ be the $r$-smoothed version of $f^*$, with variance $\sigma_r^2 = \sigma^2 + r^2$. Let $s_r$ be the score function of $f_r$, and let $\I_r$ be the Fisher information. Let $\Tilde s_r$ be any score function with $|\Tilde s_r(x)| \le \frac{2}{r} \sqrt{\log \frac{N}{\log \frac{1}{\delta}}}$ for all $x$. Let $t =\sqrt{\frac{\gamma \log \frac{N}{\I_r r^2\log \frac{1}{\delta}}}{\I_r r^2} } $ for $\gamma \ge  1$, and let $\alpha = \frac{1}{t^3 \sigma}$. Then,
\begin{align*}
    \E_{x \sim f_r}[(\Tilde s_r(x) - s_r(x))^2 \1_{|x - \mu| > t\sigma_r}] \lesssim \frac{\I_r}{\gamma}
\end{align*}
\end{lemma}
\begin{proof}
    By Chebyshev's inequality,
    $$\Pr[|x - \mu| > t \sigma_r] \le \frac{1}{t^2}$$
    So, by Lemma~\ref{lemma:score_error_bounded_in_small_set}, the claim follows.
\end{proof}
The main result of this section as follows shows that the clipped KDE score approximates the true score well.
\clippedkdeerror*
\begin{proof}
    Let $\mu, \sigma_r^2$ be the mean and variance of $f_r$. Let $t = \sqrt{\frac{\gamma \log \frac{N}{\I_r r^2\log \frac{1}{\delta}}}{\I_r r^2}}$, $\alpha = \frac{1}{t^3\sigma_r}$. Now,
    \begin{align*}
        \E_{x \sim f_r}\left[(\wh s_r^\clip(x) - s_r(x))^2 \right] &= \E_{x \sim f_r}\left[(\wh s_r^\clip(x) - s_r(x))^2 \1_{\{|x - \mu| \le t \sigma_r \text{ and } f_r(x) \ge \alpha\}}\right] \\
        &\quad + \E_{x \sim f_r}\left[(\wh s_r^\clip(x) - s_r(x))^2 \1_{\{|x - \mu| \le t \sigma_r \text{ and } f_r(x) < \alpha\}}\right]\\
        &\quad + \E_{x \sim f_r}[(\wh s_r^\clip(x) - s_r(x))^2 \1_{|x - \mu| > t\sigma_r}]
    \end{align*}
    So, by Lemmas~\ref{lem:error_estimate_large_density_clipped_kde},~\ref{lem:clipped_kde_small_density_typical}~and~\ref{lem:clipped_kde_atypical}, we have the claim.
\end{proof}

\section{Symmetrization}
In this section we show that the expectation of our symmetrized, clipped KDE score function $\wh s_r^\sym$, symmetrized around a point $\mu + \eps$ for small $\eps$ has expectation close to $\eps \I_r$, where $\I_r$ is the Fisher information of the $r$-smoothed distribution. We also show that the second moment of $\wh s_r^\sym$ is close to $\I_r$. We begin by recalling the definition of $\wh s_r^\sym$ from \eqref{eq:sym_kde_def_main_body}.
\paragraph{Definition}
Let the symmetrized, clipped KDE score, symmetrized around a point $y$ from $N$ samples be given by
\begin{equation}
\label{eq:symmetrized_kde_def}
    \wh s_r^\sym(x) = \begin{cases}
        \wh s_r^\clip(x) & x \ge y\\
        -\wh s_r^\clip(2y - x) & x < y
    \end{cases}
\end{equation}
where $\wh s_r^\clip$ is the clipped KDE score from $N$ samples, as defined in \eqref{eq:clipped_kde_def}.

First we show that the true score function centered at $-\eps$ is close to the true score centered at $0$ in $\ell_2$ distance.
\begin{lemma}
\label{lem:shifted_true_score_error}
    Let $f_r$ be an $r$-smoothed distribution with score function $s_r$ and Fisher information $\I_r$. Then, for $|\eps| \le r/60$,
    \[
        \E_{x \sim f_r}\left[(s_r(x+\eps) - s_r(x))^2 \right] \lesssim \frac{\eps^2}{r^4}
    \]
\end{lemma}
\begin{proof}
By Lemma~\ref{lem:shifted_score_characterization},
\begin{align}
\label{eq:shifted_score_error_bound}
\begin{split}
    s_r(x+\eps) - s_r(x) &= \frac{\E_{Z_r|x}\left[e^{\frac{\eps Z_r }{r^2}} \frac{Z_r - \eps}{r^2} \right]}{\E_{Z_r | x} \left[e^{\frac{\eps Z_r }{r^2}} \right]}-  \E_{Z_r | x} \left[\frac{Z_r}{r^2}\right]\\
    &= -\frac{\eps}{r^2} + \frac{\E_{Z_r | x}\left[e^{\frac{\eps Z_r}{ r^2}} (Z_r - \E_{Z_r|x}[Z_r]) \right]}{r^2 \E_{Z_r|x}\left[ e^{\frac{\eps Z_r}{r^2}} \right]}
\end{split}
\end{align}
Now,
\begin{align*}
\E_{Z_r | x}\left[e^{\frac{\eps Z_r}{ r^2}} (Z_r - \E_{Z_r|x}[Z_r])\right] &= \E_{Z_r | x}\left[(e^{\frac{\eps Z_r}{r^2}} - 1) Z_r \right] - \E_{Z_r|x}\left[e^{\frac{\eps Z_r}{r^2}} - 1\right]\E_{Z_r|x}[Z_r]\\
\end{align*}
So,
\begin{align*}
   \E_{Z_r | x}\left[e^{\frac{\eps Z_r}{ r^2}} (Z_r - \E_{Z_r|x}[Z_r])\right]^4 &\lesssim \E_{Z_r | x}\left[(e^{\frac{\eps Z_r}{r^2}} - 1) Z_r \right]^4 + \E_{Z_r|x}\left[e^{\frac{\eps Z_r}{r^2}} - 1\right]^4\E_{Z_r|x}[Z_r]^4\\
    &\lesssim \E_{Z_r|x}\left[(e^{\frac{\eps Z_r}{r^2}} - 1)^4 \right] \E_{Z_r|x}[Z_r^4] \quad \text{by Cauchy-Schwarz and Jensen's inequalities}
\end{align*}
So, we have, by Cauchy-Schwarz and Jensen's inequalities,
\begin{align}
\label{eq:shifted_error_numerator_bound}
\E_x\left[\E_{Z_r | x}\left[e^{\frac{\eps Z_r}{ r^2}} (Z_r - \E_{Z_r|x}[Z_r])\right]^4\right] &\lesssim \sqrt{\E_{Z_r}\left[(e^{\frac{\eps Z_r}{r^2}} - 1)^8 \right] \E_{Z_r}\left[ Z_r^8\right]}
\end{align}
We will now bound $\E_{Z_r}\left[(e^{\frac{\eps Z_r}{r^2}} - 1)^8 \right]$. By a Taylor expansion, when $|\eps z| \le r^2$
\begin{align*}
    e^{\frac{\eps z}{r^2}} - 1 = \frac{\eps z}{r^2} + O\left( \left( \frac{\eps z}{r^2}\right)^2\right)
\end{align*}
so that
\[
    \E_{Z_r \sim \mathcal N(0,r^2)}\left[(e^{\frac{\eps Z_r}{r^2}} - 1)^8 \1_{|\eps Z_r| \le r^2} \right] \lesssim \frac{\eps^8}{r^{16}} \E_{Z_r}\left[Z_r^8 \right] \lesssim \frac{\eps^8}{r^8}
\]
On the other hand, when $|\eps z| \ge r^2$, we have $(e^{\frac{\eps z}{r^2}} - 1)^8 \lesssim e^{\frac{8|\eps z|}{r^2}}$, meaning that
\begin{align*}
    \E_{Z_r \sim \mathcal N(0, r^2)}\left[(e^{\frac{\eps Z_r}{r^2}} - 1)^8 \1_{|\eps Z_r| > r^2} \right] &\lesssim \int_{|r^2/\eps|}^\infty \frac{1}{\sqrt{2 \pi} r} e^{ \frac{8|\eps z|}{r^2}} e^{-\frac{z^2}{2 r^2}} dz\\
    &= e^{\frac{32 \eps^2}{r^2}}\int_{|r^2/\eps|}^\infty \frac{1}{\sqrt{ 2\pi} r} e^{-\frac{(z - 8 |\eps|)^2}{2 r^2}} dz\\
    &\lesssim \Pr_{Z_r\sim \mathcal N(0,r^2)}\left[Z_r \ge r^2/|\eps| - 8 |\eps| \right]\\
    &\lesssim e^{-\frac{(|r^2/\eps| - 8 |\eps|)^2}{2 r^2}}\\
    &\lesssim \frac{\eps^8}{r^8}
\end{align*}
Also, $\E_{Z_r} [Z_r^8] \lesssim r^8$
So, we have shown in \eqref{eq:shifted_error_numerator_bound},
\begin{align*}
    \E_x\left[\E_{Z_r | x}\left[e^{\frac{\eps Z_r}{ r^2}} (Z_r - \E_{Z_r|x}[Z_r])\right]^4\right] \lesssim \frac{\eps^4}{r^4} \sqrt{\E_{Z_r} [Z_r^8]} \lesssim \eps^4
\end{align*}
Also, using Jensen's inequality
\begin{align*}
    \E_x\left[\frac{1}{\E_{Z_r | x}\left[e^{\frac{\eps Z_r}{r^2}} \right]^4} \right] &\le \E_{x}\left[e^{-4 \eps \E_{Z_r|x}[Z_r/r^2]} \right]\\
    &= \E_{x}\left[e^{-4 \eps s_r(x)} \right] \quad \text{since $s_r(x) = \E_{Z_r | x}[Z_r]/r^2$ by Lemma~\ref{lem:shifted_score_characterization}}\\
    &\le e^{8 \eps^2/r^2} \lesssim 1 \quad \text{by Lemma~\ref{lem:score-subgaussian} and since $|\eps| \le r/60$}
\end{align*}
Then, using \eqref{eq:shifted_score_error_bound},
\begin{align*}
    \E_{x \sim f_r}\left[(s_r(x+\eps) - s_r(x))^2 \right] &\lesssim \frac{\eps^2}{r^4} + \frac{1}{r^4}\E_{x}\left[\frac{\E_{Z_r | x}\left[e^{\frac{\eps Z_r}{r^2}}(Z_r - \E_{Z_r |x} [Z_r])\right]^2}{\E_{Z_r | x}[e^{\frac{\eps Z_r}{r^2}}]^2} \right]\\
    &\le \frac{\eps^2}{r^4} + \frac{1}{r^4} \sqrt{\E_x\left[\E_{Z_r | x}\left[e^{\frac{\eps Z_r}{r^2}} (Z_r - \E_{Z_r | x}[Z_r])\right]^4 \right] \E_x\left[\frac{1}{\E_{Z_r | x}\left[e^{\frac{\eps Z_r}{r^2}} \right]^4} \right]}\\
    &\lesssim \frac{\eps^2}{r^4}
\end{align*}
\end{proof}
The  next lemma shows that $\wh s_r^\sym$ approximates $s_r$ in a certain sense. Using this, we obtain that the second moment of $\wh s_r^\sym$ is close the $\I_r$.
\begin{lemma}
\label{lem:sym_kde_variance_error}
    Let $f^*$ be an arbitrary symmetric distribution with mean $\mu$, and let $f_r$ be the $r$-smoothed version of $f^*$. Let $s_r$ be the score function of $f_r$, and let $\I_r$ be the Fisher information. Let $\wh s_r^\sym$ be the symmetrized clipped Kernel Density Estimate score from $N$ samples, symmetrized around $\mu + \eps$ for $|\eps| \le r/60$, as defined in \eqref{eq:symmetrized_kde_def}. Let $\gamma > C$ be a parameter for large enough constant $C$. Then for any $r \le \sigma$ and $\frac{N}{\log \frac{1}{\delta}} \ge \left(\frac{\gamma^{5/12} \sigma}{r} \right)^{6 + \beta}$ for $\beta > 0$, if $|\eps| \le r^2 \sqrt{\frac{\I_r}{\gamma}}$, with probability $1 - \delta$,
    \[
        \E_{x \sim f_r}\left[(\wh s_r^\sym(x) - s_r(x))^2 \right]\lesssim \frac{\I_r}{\gamma} 
    \]
\end{lemma}
\begin{proof}
    By definition of $\wh s_r^\sym$, and using Lemma~\ref{lem:clipped_kde_error},
    \begin{align*}
        \E_{x \sim f_r}\left[(\wh s_r^\sym(x) - s_r(x))^2 \1_{x \ge \mu + \eps}\right] = \E_{x \sim f_r}\left[(\wh s_r^\clip(x) - s_r(x))^2 \1_{x \ge \mu + \eps}\right] \lesssim \frac{\I_r}{\gamma}
    \end{align*}
    On the other hand, by Lemmas~\ref{lem:clipped_kde_error}~and~\ref{lem:shifted_true_score_error}
    \begin{align*}
        &\E_{x \sim f_r}\left[(\wh s_r^\sym(x) - s_r(x))^2 \1_{x < \mu + \eps}\right]\\
        &= \E_{x \sim f_r}\left[(-\wh s_r^\clip(2(\mu + \eps) - x) + s_r(2 \mu - x))^2 \1_{x < \mu + \eps}\right]\\
        &\le \E_{x \sim f_r}\left[(-\wh s_r^\clip(2(\mu + \eps)- x) + s_r(2(\mu + \eps) - x))^2 \1_{x < \mu + \eps}\right] + \E_{x \sim f_r}\left[( s_r((2\mu - x) + 2 \eps) - s_r(2\mu - x))^2 \right]\\
        &\lesssim \frac{\I_r}{\gamma} + \frac{\eps^2}{r^4}
    \end{align*}
    The claim follows since $\frac{\eps^2}{r^4} \le \frac{\I_r}{\gamma}$.
\end{proof}
\symkdevariancecor*
\begin{proof}
    We have
    \begin{align*}
        \E_{x\sim f_r}\left[\wh s_r^\sym(x)^2 \right] &= \E_{x \sim f_r}\left[(s_r(x) + \wh s_r^\sym(x) - s_r(x))^2 \right]\\
        &= \E_{x \sim f_r}\left[s_r(x)^2 \right] + 2 \E_{x \sim f_r}\left[ s_r(x) (\wh s_r^\sym(x) - s_r(x))\right] + \E_{x \sim f_r}\left[(\wh s_r^\sym(x) - s_r(x))^2 \right]\\
        &\le \E_{x \sim f_r} [s_r(x)^2] + 2 \sqrt{\E_{x \sim f_r}[s_r(x)^2] \E_{x \sim f_r}[(\wh s^\sym(x) - s_r(x))^2]} + \E_{x \sim f_r}[(\wh s_r^\sym(x) - s_r(x))^2]\\
    \end{align*}
    So, by Lemma~\ref{lem:sym_kde_variance_error},
    \begin{align*}
        \E_{x \sim f_r}\left[ \wh s_r^\sym(x)^2 \right] - \I_r &\lesssim \frac{\I_r}{\sqrt{\gamma}} + \frac{\I_r}{\gamma} \lesssim \frac{\I_r}{\sqrt{\gamma}}\\
    \end{align*}
    Similarly,
    \begin{align*}
        \E_{x\sim f_r}\left[s_r(x)^2 \right] &\le \E_{x \sim f_r}\left[ \wh s_r^\sym(x)^2\right] + 2 \sqrt{\E_{x \sim f_r}[\wh s_r^\sym(x)^2] \E_{x \sim f_r}\left[(\wh s_r^\sym(x) - s_r(x))^2 \right] } + \E_{x \sim f_r}\left[ (\wh s_r^\sym(x) - s_r(x))^2\right]\\
    \end{align*}
    So, since we showed $\E_{x \sim f_r}[\wh s_r^\sym(x)^2] \lesssim \I_r$, we have
    \begin{align*}
        \E_{x \sim f_r}[s_r(x)^2] \le \E_{x \sim f_r}\left[\wh s_r^\sym(x)^2 \right] + O\left(\frac{\I_r}{\sqrt{\gamma}} \right)\\
    \end{align*}
    so that
    \begin{align*}
        \E_{x \sim f_r}\left[\wh s_r^\sym(x)^2 \right] - \I_r \gtrsim \frac{\I_r}{\sqrt{\gamma}}
    \end{align*}
    The claim follows.
\end{proof}

Finally, we show that the expectation of the symmetrized, clipped KDE score function $\wh s_r^\sym$ symmetrized around $\mu + \eps$ for small $\eps$ is close to $\eps \I_r$.
\symkdeerror*
\begin{proof}
    Since $f_r$ is symmetric around $\mu$, $f_r(x) = f_r(2 \mu - x)$. So using the definition of $\wh s_r^\sym$, we have
    \begin{align*}
        \int_{-\infty}^{\mu + \eps} f_r(x-\eps) \wh s^\sym_r(x)dx &= - \int_{-\infty}^{\mu + \eps} f_r(2 \mu - x + \eps) s_r^\clip(2(\mu+\eps) - x) dx\\
        &=\int_{\infty}^{\mu + \eps} f_r(y -  \eps) s_r^\clip(y) dy \quad \text{Substituting $y = 2(\mu + \eps) - x$}\\
        &= -\int_{\mu + \eps}^\infty f_r(y-\eps) s_r^\sym(y) dy \quad \text{since $s_r^\sym(x) = s_r^\clip(x)$ for $x \ge \mu + \eps$}
    \end{align*}
    So, we have
    \begin{align*}
        \E_{x \sim f_r}\left[\frac{f_r(x-\eps)}{f_r(x)}\wh s_r^\sym(x) \right] = \int_{-\infty}^{\mu + \eps} f_r(x-\eps) \wh s_r^\sym(x) dx + \int_{\mu+\eps}^\infty f_r(x-\eps) s_r^\sym(x) dx = 0
    \end{align*}
    So,
    \begin{align*}
        \E_{x \sim f_r}\left[\wh s_r^\sym(x) \right] = \E_{x \sim f_r}\left[\frac{f_r(x)-f_r(x-\eps)}{f_r(x)} \wh s_r^\sym(x)\right]
    \end{align*}
    Thus,
    \begin{align*}
        &\left(\E_{x \sim f_r}\left[\wh s_r^\sym(x) \right] - \eps  \E_{x \sim f_r}\left[ \wh s_r^\sym(x)^2 \right]\right)^2\\
        &= \E_{x \sim f_r}\left[\left(\frac{f_r(x) - f_r(x-\eps)}{f_r(x)} - \eps \wh s_r^\sym(x) \right) \wh s_r^\sym(x)\right]^2\\
        &= \E_{x \sim f_r}\left[\left(\frac{f_r(x) - f_r(x - \eps) - \eps f_r'(x)}{f_r(x)} + \eps(s_r(x) - \wh s_r^\sym(x))\right) \wh s_r^\sym(x)\right]^2\quad \text{since $s_r(x)= \frac{f_r'(x)}{f_r(x)}$}\\ 
        &\le \E_{x \sim f_r}\left[\left(\frac{f_r(x) - f_r(x - \eps) - \eps f_r'(x)}{f_r(x)} + \eps(s_r(x) - \wh s_r^\sym(x))\right)^2 \right]\E_{x \sim f_r}\left[\wh s_r^\sym(x)^2 \right]\\
        &\lesssim \left(\E_{x \sim f_r}\left[\left(\frac{f_r(x) - f_r(x - \eps) - \eps f_r'(x)}{f_r(x)}\right)^2 \right] + \eps^2 \E_{x \sim f_r}\left[(s_r(x) - \wh s_r^\sym(x))^2 \right] \right) \E_{x \sim f_r}[\wh s_r^\sym(x)^2]\\
        &= \left(\E_{x \sim f_r}\left[ \Delta_{-\eps}(x)^2\right] + \eps^2 \E_{x \sim f_r}\left[(s_r(x) - \wh s_r^\sym(x))^2 \right]\right)\E_{x \sim f_r}[\wh s_r^\sym(x)^2]
    \end{align*}
    where $\Delta_\eps(x) = \frac{f_r(x+\eps) - f_r(x) - \eps f_r'(x)}{f_r(x)}$. But by Lemma~\ref{lem:Delta_eps_bound}, $\E_{x \sim f_r}\left[\Delta_{-\eps}(x)^2 \right] \lesssim \frac{\eps^4}{r^4} \lesssim \eps^2 \frac{\I_r}{\gamma}$ (since $|\eps| \le r^2 \sqrt{\frac{\I_r}{\gamma}}$). So, by Lemma~\ref{lem:sym_kde_variance_error} and Lemma~\ref{cor:sym_kde_variance}, we have
    \begin{align*}
        \left|\E_{x \sim f_r}[\wh s_r^\sym(x)] - \eps\E_{x \sim f_r}[\wh s_r^\sym(x)^2]\right| &\lesssim \sqrt{\left(\eps^2 \frac{\I_r}{\gamma} + \eps^2 \frac{\I_r}{\gamma}\right) \I_r}\\
        &\lesssim \frac{\eps \I_r}{\sqrt{\gamma}}
    \end{align*}
    Then, using Lemma~\ref{cor:sym_kde_variance} once again on $\E_{x \sim f_r}[s_r^\sym(x)^2]$ on the LHS, the claim follows.
\end{proof}
\section{Estimating $\I_r$}
In this section, we show that $\wh \I_r = \E_{x \sim \wh f_r}[\wh s_r^\sym(x)^2]$ provides a good estimate of $\I_r$.
\begin{lemma}
\label{lem:wh_f_bin_close}
    Let $\wh f_r$ be the kernel density estimate of $f_r$ from $N$ samples $Y_1, \dots, Y_N \sim f^*$ as defined in \eqref{eq:kde_def}. Let $x, \eps, N$ be such that $N \ge \frac{12 \log \frac{2}{\delta}}{f_r(x) r}$ and $\frac{16|\eps|}{r} \sqrt{\log \frac{1}{r f_r(x)}} \le 1$. We have that with probability $1-\delta$, for all $|\zeta| \le |\eps| $ simultaneously,
    \begin{align*}
        |\wh f_r(x+\zeta) - f_r(x+\zeta)| \lesssim \sqrt{\frac{f_r(x+\zeta) \log \frac{2}{\delta}}{N r}} + \frac{|\eps| f_r(x+\zeta)}{r} \sqrt{\log \frac{2}{f_r(x+\zeta) r}}
    \end{align*}
\end{lemma}
\begin{proof}
    By Lemma~\ref{lem:pointwise_density_estimate}, with probability $1 - \delta$,
    \[
        |\wh f_r(x) - f_r(x)| \le \sqrt{\frac{3 f_r(x) \log \frac{2}{\delta}}{N r}}
    \]
    Note that since $N \ge \frac{12 \log \frac{2}{\delta}}{f_r(x) r}$, the RHS above is at most $\frac{f_r(x)}{2}$.
    
    Also, by Lemma~\ref{lem:density_close},
    \begin{align*}
        |f_r(x+\zeta) - f_r(x)| \le \frac{10 |\eps| f_r(x)}{r} \sqrt{ \log \frac{1}{f_r(x) r}}
    \end{align*}
    Now, since $\frac{8 |\eps|}{r} \sqrt{ \log \frac{1}{r \wh f_r(x)}} \le \frac{8|\eps|}{r} \sqrt{\log \frac{2}{r f_r(x)}} \le \frac{16 |\eps|}{r} \sqrt{\log \frac{1}{r f_r(x)}} < 1$, by Lemma~\ref{lem:density_close},
    \begin{align*}
        |\wh f_r(x+\zeta) - \wh f_r(x)| &\le \frac{10|\eps| \wh f_r(x)}{r} \sqrt{ \log \frac{1}{\wh f_r(x) r}}\\
        &\le \frac{15 |\eps| f_r(x)}{r} \sqrt{\log \frac{2}{f_r(x) r}}
    \end{align*}
    So, putting everything together
    \begin{align*}
        |\wh f_r(x+\zeta) - f_r(x+\zeta)| &\le |\wh f_r(x+\zeta) - \wh f_r(x)| + |\wh f_r(x) - f_r(x)| + |f_r(x) - f_r(x+\zeta)|\\
        &\le \sqrt{\frac{3f_r(x) \log \frac{2}{\delta}}{N r}} + \frac{30|\eps| f_r(x)}{r} \sqrt{\log \frac{2}{f_r(x) r}}
    \end{align*}
    Now, since $\frac{16 |\eps|}{r} \sqrt{\log \frac{1}{r f_r(x)}} \le 1$ so that by Lemma~\ref{lem:density_close}
    \[
        |f_r(x+\zeta) - f_r(x)| \le \frac{5}{8} f_r(x)
    \]
    The claim follows.
\end{proof}

\begin{lemma}
\label{lem:generic_kde_density_close_large}
    Let $f^*$ be an arbitrary distribution with mean $\mu$ and variance $\sigma^2$ and let $f_r$ be the $r$-smoothed version of $f^*$, with variance $\sigma_r^2 = \sigma^2 + r^2$. Let $\wh f_r$ be the kernel density estimate of $f_r$ from $N$ samples $Y_1, \dots, Y_N \sim f^*$ as defined in \eqref{eq:kde_def}. Let $0<\alpha\le \frac{1}{\sqrt{2 \pi} r} $ and let $N \ge \frac{12 \log \left(\frac{2}{\delta}\left(\frac{2 t \sigma_r \sqrt{\alpha N}}{\sqrt{r}} + 1 \right)\right) + 400 \log \frac{1}{\alpha r}}{\alpha r}$. Then, with probability $1 - \delta$, for all $x$ such that $|x - \mu| \le t\sigma_r$ and $f_r(x) \ge \alpha$ simultaneously,
    \begin{align*}
        |\wh f_r(x) - f_r(x)|\lesssim f_r(x)\sqrt{\frac{1}{\alpha r N} \log \left(\frac{2}{\delta} \left(\frac{2 t \sigma_r \sqrt{\alpha N}}{\sqrt{r}} + 1 \right) \right) \log \frac{2}{\alpha r}} 
    \end{align*}
\end{lemma}
\begin{proof}
    Consider contiguous intervals of length $\eps$ starting from $\mu - t \sigma_r$ so that the last interval covers $\mu + t \sigma_r$, and let $S$ be the set of the smallest $y$ such that $f_r(y) \ge \alpha$ in each of these intervals, if one exists. Note that $|S| \le \frac{2 t \sigma_r}{\eps} + 1$. Then, we have that
    \begin{align*}
        \{x : |x - \mu| \le t \sigma_r \text{ and } f_r(x) \ge \alpha\} \subseteq \{[y-\eps, y+\eps] | y \in S\}
    \end{align*}
    Now for $\eps = \sqrt{\frac{r}{\alpha N}}$ and $y \in S$, since $N\ge \frac{12 \log \left(\frac{2}{\delta}\left(\frac{2 t \sigma_r \sqrt{\alpha N}}{\sqrt{r}} + 1 \right)\right)}{\alpha r}\ge \frac{12 \log\frac{2|S|}{\delta}}{f_r(y) r}$ and $\frac{16|\eps|}{r} \sqrt{\log \frac{1}{r f_r(y)}} = \frac{16}{\sqrt{\alpha r N}}\sqrt{\log \frac{1}{f_r(y) r}} \le \frac{16}{\sqrt{\alpha r N}}\sqrt{\log \frac{1}{\alpha r}} \le 1$, we have by Lemma~\ref{lem:wh_f_bin_close} that with probability $1 - \frac{\delta}{|S|}$, for all $|\zeta| \le |\eps|$ simultaneously,
    \begin{align*}
        \left| \wh f_r(y + \zeta ) - f_r(y+\zeta)\right| &\lesssim \sqrt{\frac{ f_r(y+\zeta)\log \frac{2|S|}{\delta}}{Nr}} + f_r(y+\zeta)\sqrt{\frac{1}{\alpha Nr} \log \frac{2}{\alpha r}}\\
        &\lesssim f_r(y+\zeta) \sqrt{\frac{1}{\alpha N r} \log \left(\frac{2}{\delta} \left(\frac{2 t \sigma_r \sqrt{\alpha N}}{\sqrt{r}} + 1 \right) \right) \log \frac{2}{\alpha r}} \\
        &\quad \text{since $f_r(y+\zeta) \gtrsim \alpha$ and $\alpha \le \frac{1}{\sqrt{2 \pi} r}$ so that $\log \frac{2}{\alpha r} > 1$}
    \end{align*} 
    So by a union bound, with probability $1 - \delta$, for all $x$ such that $|x - \mu| \le t \sigma_r$ and $f_r(x) \ge \alpha$ simultaneously,
    \begin{align*}
        |\wh f_r(x) - f_r(x)|\lesssim f_r(x) \sqrt{\frac{1}{\alpha N r} \log \left(\frac{2}{\delta} \left(\frac{2 t \sigma_r \sqrt{\alpha N}}{\sqrt{r}} + 1 \right) \right) \log \frac{2}{\alpha r}} 
    \end{align*}
\end{proof}
\begin{lemma}
\label{lem:kde_density_hoeffding}
    Let $S$ be a set and let $\wh f_r$ be the kernel density estimate of $f_r$ from $N$ samples $Y_1, \dots, Y_N \sim f^*$. Then, with probability $1 - \delta$,
    \[
        \left|\Pr_{\wh f_r}[S] - \Pr_{f_r}[S] \right| \lesssim \sqrt{\frac{\log \frac{2}{\delta}}{N}}
    \] 
\end{lemma}
\begin{proof}
    \begin{align*}
        \E_{Y_i}\left[\left[\Pr_{\wh f_r}[S]\right]\right] = \E_{Y_i}\left[\int_S \wh f_r(x) dx \right] = \int_S\E_{Y_i}\left[\wh f_r(x) \right] dx = \int_S f_r(x) dx = \Pr_{f_r}[S]
    \end{align*}
    Furthermore $0 \le \Pr_{\wh f_r}\left[S \right]\le 1$. So, by Hoeffding's inequality, with probability $1 - \delta$,
    \begin{align*}
        \left| \Pr_{\wh f_r}[S] - \Pr_{f_r}[S] \right|\lesssim \sqrt{\frac{\log \frac{2}{\delta}}{N}}
    \end{align*}
\end{proof}

\begin{lemma}
\label{lem:small_prob_sets_kde_density}
    Let $\gamma \ge C$ for large enough constant $C$, and let $\frac{N}{\log \frac{1}{\delta}} \ge \left(\frac{\gamma}{r^4 \I_r^2}\right)^{1 + \alpha}$ for small constant $\alpha > 0$. Let $\Tilde s$ be any function with $|\Tilde s_r(x)| \le \frac{2}{r} \sqrt{\log \frac{N}{\log \frac{1}{\delta}}}$ for all $x$. Let $S$ be a set with $\Pr_{f_r}[S] \lesssim \frac{1}{t^2}$ for $t= \gamma^{1/4}\sqrt{\frac{\log \frac{N}{\I_r r^2 \log \frac{1}{\delta}}}{\I_r r^2}}$. Let $\wh f_r$ be the Kernel density estimate of $f_r$ as defined in \eqref{eq:kde_def} from $N$ samples. Then, we have
    \begin{align*}
       \int_S \left(\wh f_r(x) - f_r(x)\right) \Tilde s_r(x)^2 dx \lesssim \frac{\I_r}{\sqrt{\gamma}}
    \end{align*}
\end{lemma}
\begin{proof}
    Since $|\Tilde s_r(x)|\le \frac{2}{r} \sqrt{\log \frac{N}{\log \frac{1}{\delta}}}$ and $\Pr_{f_r}[S] \le \frac{1}{t^2}$, we have
    \begin{align*}
         \int_S \left(\wh f_r(x) - f_r(x)\right) \Tilde s_r(x)^2 dx &\le \int_S \wh f_r(x) \Tilde s_r(x)^2 dx\\
         &\le \frac{4}{r^2} \log \left(\frac{N}{\log \frac{1}{\delta}}\right) \Pr_{\wh f_r} [S]\\
         &\lesssim \frac{4}{r^2} \log \left(\frac{N}{\log \frac{1}{\delta}}\right) \left(\Pr_{f_r}[S] + \sqrt{\frac{\log \frac{2}{\delta}}{N}} \right)\quad \text{by Lemma~\ref{lem:kde_density_hoeffding}}\\
         &\lesssim \frac{1}{t^2 r^2} \log \left(\frac{N}{\log \frac{1}{\delta}} \right) + \frac{1}{r^2} \sqrt{\frac{\log \frac{1}{\delta}}{N}} \log \left(\frac{N}{\log \frac{1}{\delta}} \right)\\
    \end{align*}
    Now,
    \begin{align*}
         \frac{1}{t^2 r^2} \log \frac{N}{\log \frac{1}{\delta}}
        &= \frac{\I_r}{\sqrt{\gamma} \log \frac{N}{\I_r r^2 \log \frac{1}{\delta}}} \log \left( \frac{N}{\log \frac{1}{\delta}} \right) \quad \text{by our setting of $t$}\\
        &\lesssim \frac{\I_r}{\sqrt{\gamma} \log \frac{N}{\I_r r^2\log \frac{1}{\delta}}} \log \left(\frac{N}{\I_r r^2 \log \frac{1}{\delta}} \right) \quad \text{since $\I_r \le \frac{1}{r^2}$}\\
        &\lesssim \frac{\I_r}{\sqrt{\gamma}}
    \end{align*}
    Also,
    \begin{align*}
        \frac{1}{r^2}\sqrt{\frac{\log \frac{2}{\delta}}{N}} \log \left( \frac{N}{\log \frac{1}{\delta}}\right)
        &\le \frac{1}{r^2} \left( \frac{\log \frac{2}{\delta}}{N}\right)^{\frac{1}{2} - \frac{\alpha}{4}}\\
        &\lesssim \frac{\I_r}{\sqrt{\gamma}}\quad \text{since $\frac{N}{\log \frac{2}{\delta}} \ge \left(\frac{\gamma}{r^4 \I_r^2} \right)^{1 + \alpha}$}
    \end{align*}
    So, the claim follows.
\end{proof}

\begin{lemma}
\label{lem:large_density_kde_instantiated}
Let $\wh f_r$ be the kernel density estimate of $f_r$ from $N$ samples, as defined in \eqref{eq:kde_def}. Let $\gamma \ge C$ for large constant $C \ge 1$ be a parameter. Suppose $\Tilde s_r$ is a function such that $\E_{x \sim f_r}[\Tilde s_r(x)] \lesssim \I_r$.
Let $t = \gamma^{1/4} \sqrt{\frac{\log \frac{N}{\I_r r^2 \log \frac{1}{\delta}}}{\I_r r^2}}$, $\alpha = \frac{1}{t^3 \sigma_r}$. Then for any $r \le \sigma$ and $\frac{N}{\log \frac{1}{\delta}} \ge \left(\gamma^{5/12} \frac{\sigma}{r} \right)^{6+\beta} $ for some small constant $\beta > 0$, with probability $1 - \delta$, we have
\begin{align*}
    \int_{-\infty}^\infty \left|\wh f_r(x) - f_r(x) \right|\Tilde s_r(x)^2 \1_{\{|x - \mu| \le t \sigma_r \text{ and } f_r(x)\ge \alpha\}} dx \lesssim \frac{\I_r}{\sqrt{\gamma}}
\end{align*}
\end{lemma}
\begin{proof}
    This proof is similar to the proof of Lemma~\ref{lem:error_estimate_large_density_instantiated}. First note that since $r \le \sigma$,
    \[
        \sigma_r^2 = \sigma^2 + r^2 \le 2 \sigma^2
    \]
    Note also that $\frac{N}{\log \frac{1}{\delta}} \ge 1$ since $\gamma \ge 1$ and $r \le \sigma$. So our setting of $N$ implies WLOG
    \begin{align*}
        \frac{N}{\log \frac{1}{\delta}} \ge \left(\frac{\gamma^{5/12} \sigma \log \frac{N}{\log \frac{1}{\delta}}}{r} \right)^6
    \end{align*}
    or
    \begin{align}
    \label{eq:sigma_over_r_density}
        \frac{\sigma}{r} \le \left(\frac{N}{\log \frac{1}{\delta}} \right)^{1/6}\cdot  \frac{1}{\gamma^{5/12} \log \frac{N}{\log \frac{1}{\delta}}} \le \left(\frac{N}{\log \frac{1}{\delta}} \right)^{1/6}
    \end{align}
    We will first check that this $N$ satisfies the condition required to invoke Lemma~\ref{lem:generic_kde_density_close_large} that $N \ge\frac{12 \log \left(\frac{2}{\delta}\left(\frac{2 t \sigma_r \sqrt{\alpha N}}{\sqrt{r}} + 1 \right)\right) + 400 \log \frac{1}{\alpha r}}{\alpha r} $. To do this, we individually upper bound $\frac{1}{\alpha r}$ and $\frac{2 t\sigma_r \sqrt{\alpha N}}{\sqrt{r}}$. We have,
    \begin{align*}
        \frac{1}{\alpha r} &= \frac{\sigma_r\gamma^{3/4}}{r} \left(\frac{\log \frac{N}{\I_r r^2 \log \frac{1}{\delta}}}{\I_r r^2} \right)^{3/2} \quad \text{since $\alpha = \frac{1}{t^3 \sigma_r}$ and $t = \gamma^{1/4}\sqrt{\frac{\log \frac{N}{\I_r r^2 \log \frac{1}{\delta}}}{\I_r r^2}}$}\\
        &\le \left( \frac{\sigma_r}{r}\right)^4 \gamma^{3/4} \log^{3/2} \frac{N \sigma_r^2}{r^2 \log \frac{1}{\delta}} \quad \text{since $\I_r \ge \frac{1}{\sigma_r^2}$}\\
        &\le \left(\frac{2 \sigma}{r}\right)^4 \gamma^{3/4} \log^{3/2} \frac{2 N \sigma^2}{r^2 \log \frac{1}{\delta}} \quad \text{since $\sigma_r^2 \le 2 \sigma^2$}\\
        &\le 16 \frac{N^{4/6}}{\gamma^{5/3} \log^4 \frac{N}{\log \frac{1}{\delta}} \log^{4/6} \frac{1}{\delta}} \cdot \gamma^{3/4} \log^{3/2}\left( 2\left(\frac{N}{\log \frac{1}{\delta}}\right)^{4/3}\right) \quad \text{by \eqref{eq:sigma_over_r_density}}\\
        &\le \frac{16N}{\gamma^{5/2} \log^4 \frac{N}{\log \frac{1}{\delta}} \log\frac{1}{\delta}}\gamma^{3/4} \log^{3/2}\left( 2\left(\frac{N}{\log \frac{1}{\delta}}\right)^{4/3}\right) \quad \text{since $\frac{N}{\log \frac{1}{\delta}} \ge \gamma^{5/2}$}\\
        &\le \frac{N}{\gamma^{3/2} \log^2 \frac{N}{\log \frac{1}{\delta}} \log \frac{1}{\delta}} \quad \text{since $\gamma \ge C$ for large enough constant $C$}
    \end{align*}
    To further justify the last line above, observe that $\frac{\gamma^{3/4}}{\gamma^{5/2}} \le \frac{1}{\gamma^{3/2} \cdot \gamma^{1/4}}$, and that for large enough constant $C$, since $\gamma \ge C$, $\gamma^{1/4}$ can be made larger than any fixed constant. Also note that $\log^{3/2}\left(2 \left(\frac{N}{\log \frac{1}{\delta}} \right)^{4/3} \right) \le \log^2 \frac{N}{\log \frac{1}{\delta}}$ for large enough constant $C$ since $\frac{N}{\log \frac{1}{\delta}} \ge \gamma \ge C$. So the inequality follows.
    Next we bound $\frac{2 t \sigma_r \sqrt{\alpha N}}{\sqrt{r}}$.
    \begin{align*}
        \frac{2 t \sigma_r\sqrt{\alpha N}}{\sqrt{r}} &= 2\sqrt{\frac{N \sigma_r}{t r}} \quad \text{since $\alpha = \frac{1}{t^3 \sigma_r}$}\\
        &= 2 \sqrt{\frac{N \sigma_r}{r \gamma^{1/4} \sqrt{\frac{\log \frac{N}{\I_r r^2 \log \frac{1}{\delta}}}{\I_r r^2}}}} \quad \text{since $t=\gamma^{1/4} \sqrt{\frac{\log \frac{N}{\I_r r^2 \log \frac{1}{\delta}}}{\I_r r^2}}$}\\
        &\le 4 \sqrt{\frac{N \sigma}{r \gamma^{1/4} \sqrt{\log 
 \frac{N}{\log \frac{1}{\delta}}}}} \quad \text{since $\I_r \le \frac{1}{r^2}$ and $\sigma_r^2 \le 2 \sigma^2$}\\
        &\le 4 \sqrt{\frac{N}{\gamma^{1/4} \sqrt{\log \frac{N}{\log \frac{1}{\delta}}}} \cdot \left(\frac{N}{\log \frac{1}{\delta}} \right)^{1/6}} \quad \text{by \eqref{eq:sigma_over_r_density}}\\
        &\le 4 N \cdot \left(\frac{N}{\log \frac{1}{\delta}} \right)^{1/12} \quad \text{since $\gamma \ge 1$, $\frac{N}{\log \frac{1}{\delta}} \ge 1$}
    \end{align*}
    So, we can now check the condition required to invoke Lemma~\ref{lem:generic_kde_density_close_large}. We have
    \begin{align*}
        &\frac{12 \log \left(\frac{2}{\delta}\left(\frac{2 t \sigma_r \sqrt{\alpha N}}{\sqrt{r}} + 1 \right)\right) + 400 \log \frac{1}{\alpha r}}{\alpha r}\\
        &\le\left(12 \log \left(\frac{2}{\delta}\left(4 N \cdot \left(\frac{N}{\log \frac{1}{\delta}} \right)^{1/12} + 1 \right)\right) + 400 \log \frac{N}{\log \frac{1}{\delta}}\right)\left( \frac{N}{\gamma^{3/2} \log^2\frac{N}{\log \frac{1}{\delta}} \log \frac{1}{\delta}}\right)\\
        &\le N \quad \text{since $\gamma \ge C$}
    \end{align*}
    So, by Lemma~\ref{lem:generic_kde_density_close_large}, we have
    \begin{align*}
         |\wh f_r(x) - f_r(x)|\lesssim f_r(x)\sqrt{\frac{1}{\alpha r N} \log \left(\frac{2}{\delta} \left(\frac{2 t \sigma_r \sqrt{\alpha N}}{\sqrt{r}} + 1 \right) \right) \log \frac{2}{\alpha r}}
    \end{align*}
    We will show that the RHS above is bounded by $O\left(\frac{f_r(x)}{\sqrt{\gamma}}\right)$.
    Since we showed that $\frac{1}{\alpha r} \le \frac{N}{\gamma^{3/2} \log^2 \frac{N}{\log \frac{1}{\delta} \log \frac{1}{\delta}}}$, we have that
    \[
        \frac{1}{r N} \le \frac{\alpha}{\gamma^{3/2} \log^2 \frac{N}{\log \frac{1}{\delta}} \log \frac{1}{\delta}}
    \]
    So, plugging this into the RHS above, along with our bounds for $\frac{2 t \sigma_r \sqrt{\alpha N}}{\sqrt{r}}$ and $\frac{1}{\alpha r}$, we have,
    \begin{align*}
        &|\wh f_r(x) - f_r(x)| \\
        &\lesssim \frac{f_r(x)}{\sqrt{\gamma}} \cdot \sqrt{\frac{1}{\gamma^{1/2} \log^2 \frac{N}{\log \frac{1}{\delta}} \log \frac{1}{\delta}} \log\left(\frac{2}{\delta}\left(4 N \cdot \left(\frac{N}{\log \frac{1}{\delta}} \right)^{1/12} \right) + 1 \right) \log \left(\frac{N}{\gamma^{3/2} \log^2 \frac{N}{\log \frac{1}{\delta}} \log \frac{1}{\delta}} \right) }\\
        &\lesssim \frac{f_r(x)}{\sqrt{\gamma}}\\
    \end{align*}
    So finally,
    \begin{align*}
        &\int_{-\infty}^\infty \left|\wh f_r(x) - f_r(x) \right|\Tilde s_r(x)^2 \1_{\{|x - \mu| \le t \sigma_r \text{ and } f_r(x)\ge \alpha\}} dx\\
        &\lesssim \frac{1}{\sqrt{\gamma}}\int_{-\infty}^\infty f_r(x) \Tilde s_r(x)^2 \1_{\{|x - \mu| \le t \sigma_r \text{ and } f_r(x)\ge \alpha\}} dx\\
        &\le \frac{1}{\sqrt{\gamma}} \E_{x \sim f_r}[\Tilde s_r(x)^2]\\
        &\lesssim\frac{\I_r}{\sqrt{\gamma}} \quad \text{since $\E_{x \sim f_r}[\Tilde s_r(x)^2]\lesssim \I_r$ by assumption}
    \end{align*}
\end{proof}
\fisherinfoestimation*
\begin{proof}
    We have
    \begin{align*}
        \E_{x \sim \wh f_r}\left[\Tilde s_r(x)^2 \right] = \E_{x \sim f_r}\left[\Tilde s_r(x)^2 \right] + \int_{-\infty}^\infty \left(\wh f_r(x) - f_r(x)\right) \Tilde s_r(x)^2 dx
    \end{align*}
    So, by our assumption,
    \begin{align*}
        \left|\E_{x \sim \wh f_r}[\Tilde s_r(x)^2] - \I_r \right| \lesssim \frac{\I_r}{\sqrt{\gamma}} + \int_{-\infty}^\infty \left(\wh f_r(x) - f_r(x)\right) \Tilde s_r(x)^2 dx
    \end{align*}
    It remains to bound the integral in the RHS above by $O\left(\frac{\I_r}{\sqrt{\gamma}} \right)$. Let $t = \gamma^{1/4} \sqrt{\frac{\log \frac{N}{\I_r r^2 \log \frac{1}{\delta}}}{\I_r r^2}}$, $\alpha = \frac{1}{t^3 \sigma_r}$. We have
    \begin{align*}
        \int_{-\infty}^\infty \left(\wh f_r(x) - f_r(x) \right)  \Tilde s_r(x)^2 dx &= \int_{-\infty}^\infty \left(\wh f_r(x) - f_r(x) \right)  s_r(x)^2 \1_{\{|x - \mu| \le t \sigma_r \text{ and } f_r(x) \ge \alpha\}}dx \\
        &\quad+ \int_{-\infty}^\infty \left(\wh f_r(x) - f_r(x) \right)  \Tilde s_r(x)^2 \1_{\{|x - \mu| \le t \sigma_r \text{ and } f_r(x) < \alpha\}}dx\\
        &\quad + \int_{-\infty}^\infty \left(\wh f_r(x) - f_r(x) \right) \Tilde s_r(x)^2 \1_{\{|x - \mu| > t \sigma_r \}}dx \\
    \end{align*}
    By Lemma~\ref{lem:large_density_kde_instantiated}, the first term in the RHS above is bounded by $\frac{\I_r}{\sqrt{\gamma}}$. To bound the other two terms, note that
    \begin{align*}
        \Pr_{x \sim f_r}\left[|x - \mu| \le t \sigma_r \text{ and } f_r(x) < \alpha\right] \le \alpha t \sigma_r \lesssim \frac{1}{t^2}
    \end{align*}
    and by Chebyshev's inequality,
    \begin{align*}
        \Pr_{x \sim f_r}\left[|x - \mu| > t \sigma_r \right] \le \frac{1}{t^2}
    \end{align*}
    Also, for small constant $\beta > 0$,
    \begin{align*}
        \left( \frac{\gamma}{r^4 \I_r^2}\right)^{1+\beta} &\le \left(\frac{\gamma \sigma_r^4}{r^4} \right)^{1+\beta} \quad \text{since $\I_r \ge \frac{1}{\sigma_r^2}$}\\
        &\le \left(\frac{4\gamma \sigma^4}{r^4} \right)^{1+\beta} \quad \text{since $\sigma_r^2 = \sigma^2 + r^2 \le 2 \sigma^2$}\\
        & \le \left(\gamma^{5/12} \frac{\sigma}{r} \right)^{6+\beta} \quad \text{since $\gamma \ge C$ for large enough constant $C$, and $r \le \sigma$}\\
        &\le \frac{N}{\log \frac{1}{\delta}}
    \end{align*}
    So, the conditions of Lemma~\ref{lem:small_prob_sets_kde_density} hold, and applying it to the second and third term in the RHS above gives the claim.
\end{proof}

\section{Local Estimation}
In this section, we describe our local estimation procedure, which takes a symmetrized and clipped KDE score function along with symmetrization point $\mu_1 = \mu + \eps$, and produces a refined estimate $\hat \mu$ of the mean.

\setcounter{algorithm}{0}
\begin{algorithm}[H]\caption{Local Estimation}
\vspace*{3mm}
\paragraph{Input Parameters:}
\begin{itemize}
\item $n$ samples $x_1, \dots, x_{n} \sim f^*$, the symmetrized and clipped KDE score function $\wh s_r^\sym$, symmetrization point $\mu_1$, Fisher information estimate $\wh \I_r$
\end{itemize}
\begin{enumerate}
\item For each sample $x_i$, compute a perturbed sample $x'_i = x_i + \mathcal N(0,r^2)$ where all the Gaussian noise are drawn independently across all the samples.
\item Compute $\hat \eps = \frac{1}{\wh \I_r n}\sum_{i=1}^{n} \wh s_r^\sym(x_i')$. Return $\hat \mu = \mu_1 - \hat \eps$.
\end{enumerate}
\end{algorithm}
\localestimationassumption*
\localestimationlemma*
\begin{proof}
    Let $\mu_1 = \mu+\eps$. First, since by Property~\ref{property:local_estimation_empirical_mean},
    \[
    \frac{1}{\wh \I_r}|\wh s_r^\sym(x)| \le \frac{2}{r \wh \I_r} \sqrt{\log \frac{n}{\xi \log\frac{\xi}{\delta}}} \le \left(1 + O\left( \frac{1}{\sqrt{\gamma}}\right) \right)\frac{2}{r \I_r} \sqrt{\log \frac{n}{\xi \log\frac{\xi}{\delta}}}
    \] for all $x$, by Bernstein's inequality, the estimate $\hat \eps$ satisfies that with probability $1-\delta$,
    
    \begin{align*}
        \left|\hat \eps -\frac{1}{\wh \I_r} \E_{x \sim f_r}\left[ \wh s_r^\sym(x)\right] \right| \le \frac{1}{\wh \I_r}\sqrt{\E_{x\sim f_r}\left[\wh s_r^\sym(x)^2 \right]} \sqrt{\frac{2 \log \frac{2}{\delta}}{n}} +O\left(\frac{ \sqrt{\log \frac{n}{\xi\log \frac{\xi}{\delta}}}}{r\I_r} \cdot \frac{\log \frac{2}{\delta}}{n}\right)
    \end{align*}
    Since $|\wh \I_r - \I_r| \lesssim \frac{\I_r}{\sqrt\gamma}$, for $\gamma > C$ for sufficiently large constant $C$, we have
    \begin{align*}
        \left|\frac{1}{\I_r} \E_{x \sim f_r}\left[ \wh s_r^\sym(x) \right] - \frac{1}{\wh\I_r} \E_{x \sim f_r}\left[ \wh s_r^\sym(x) \right] \right| \lesssim \frac{1}{ \I_r\sqrt{\gamma}} \left|\E_{x \sim f_r}\left[ \wh s_r^\sym(x)\right] \right| \lesssim\frac{\eps}{\sqrt{\gamma}}
    \end{align*}
    Combining this with the above and the fact that $|\wh \I_r - \I_r| \lesssim \frac{\I_r}{\sqrt\gamma}$ yields
    \begin{align*}
    &\left|\hat \eps -\frac{1}{\I_r} \E_{x \sim f_r}\left[ \wh s_r^\sym(x)\right] \right|\\
    &\le\left(1 + O\left(\frac{1}{\sqrt{\gamma}} \right)\right) \frac{1}{\I_r}\sqrt{\E_{x\sim f_r}\left[\wh s_r^\sym(x)^2 \right]} \sqrt{\frac{2 \log \frac{2}{\delta}}{n}} + O\left(\frac{ \sqrt{\log \frac{n}{\xi\log \frac{\xi}{\delta}}}}{r\I_r} \cdot \frac{\log \frac{2}{\delta}}{n}\right) + O\left(\frac{\eps}{\sqrt{\gamma}}\right)
    \end{align*}
    Then, combined with Property~\ref{property:local_estimation_empirical_mean} and the triangle inequality, this implies that with probability $1 - \delta$,
    \begin{align*}
        \left| \hat \eps - \eps\right| \le \left(1 + O\left(\frac{1}{\sqrt{\gamma}} \right)\right)\sqrt{\frac{2 \log \frac{2}{\delta}}{n \I_r}}  + O\left(\frac{ \sqrt{\log \frac{n}{\xi\log \frac{\xi}{\delta}}}}{r\I_r} \cdot \frac{\log \frac{2}{\delta}}{n}\right)+ O\left(\frac{\eps}{\sqrt{\gamma}}\right)
    \end{align*}
    So, since $\mu_1 = \mu + \eps$ and $\hat \mu = \mu + \hat \eps$, we have the claim.
\end{proof}

\section{Global Estimation}
In this section, we describe our global estimation procedure and show that it provides a good estimate of the mean. It uses a small number of samples to compute an initial estimate $\mu_1$ of $\mu$, and uses another small set of samples to compute the symmetrized, clipped KDE score function $\wh s_r^\sym$ symmetrized around $\mu_1$. It then uses our local estimation procedure to produce the final estimate $\hat \mu$.
\setcounter{algorithm}{1}
\begin{algorithm}[H]\caption{Global Estimation}
    \vspace*{3mm}
    \paragraph{Input parameters:}
    \begin{itemize}
        \item Failure probability $\delta$, Samples $x_1, \dots, x_n \sim f^*$, smoothing parameter $r$, approximation parameter $\xi > 0$.
    \end{itemize}
    \begin{enumerate}
        \item First, use the first $n/\xi$ samples to compute an initial estimate $\mu_1$ of the mean $\mu$ by using the Median-of-pairwise-means estimator in Lemma~\ref{lem:median_of_means}.
        \item Use the next $n/\xi$ samples to compute the kernel density estimate $\wh f_r$ of $f_r$ (as defined in \eqref{eq:kde_def}), along with the associated symmeterized, clipped KDE score $\wh s_r^\sym$ (as defined in \eqref{eq:symmetrized_kde_def}), clipped at $\frac{2}{r} \sqrt{\log \frac{n}{\xi \log \frac{\xi}{\delta}}}$ and symmetrized around the initial estimate $\mu_1$. Compute the Fisher information estimate $\wh \I_r = \E_{x\sim \wh f_r}\left[\wh s_r^\sym(x)^2 \right]$.
        \item Run Algorithm~\ref{alg:local_estimation_empirical_mean} using the remaining $n - \frac{2n}{\xi}$ samples, and return the final estimate $\hat \mu$.
    \end{enumerate}
\end{algorithm}

\globalestimationlemma*
\begin{proof}
    Let $\eps = \mu_1 - \mu$ where $\mu_1$ is our median of means estimate in Step 1. First, note that by Lemma~\ref{lem:median_of_means}, $\eps$ satisfies that with probability $1 - \delta/\xi$,
    \begin{align*}
        |\eps| &\lesssim \sigma \cdot \sqrt{\frac{ \xi \log \frac{\xi}{\delta}}{n}}\\
    \end{align*}
    We condition on Step $1$ succeeding so that the above holds. To obtain bounds on the expectation and variance of $\wh s^\sym$, we will now check that the following conditions required to invoke Lemma~\ref{cor:sym_kde_variance} and Lemma~\ref{lem:sym_kde_mean} hold:
    \begin{itemize}
        \item $|\eps| \le r/60$
        \item $|\eps| \le r^2 \sqrt{\frac{\I_r}{\gamma}}$
    \end{itemize}
    Note that $\frac{\log \frac{1}{\delta}}{n} \le 1$ and $\gamma \ge 1$. Note also that $\xi \le \gamma$. So, by our setting of $n$, 
    \[
        r \ge \sigma \gamma^{5/12} \xi^{1/6} \left(\frac{ \log \frac{1}{\delta}}{n}\right)^{1/6} \gg O\left(\sigma\cdot \sqrt{\frac{\xi \log \frac{\xi}{\delta}}{n}} \right)\ge |\eps|
    \]
    so that $|\eps| \le r/60$.
    Similarly
    \begin{align*}
        r^2 \sqrt{\frac{\I_r}{\gamma}} &\ge \sigma^2 \gamma^{5/6} \xi^{1/3} \left(\frac{\log \frac{1}{\delta}}{n} \right)^{1/3} \sqrt{\frac{\I_r}{\gamma}}\\
        &\ge \frac{\sigma^2\gamma^{5/6} \xi^{1/3}}{\sqrt{\sigma^2 + r^2}} \left(\frac{\log \frac{1}{\delta}}{n} \right)^{1/3} \quad \text{since $\I_r \ge \frac{1}{\sigma^2 + r^2}$}\\
        &\ge \frac{\sigma}{2} \gamma^{5/6} \xi^{1/3} \left(\frac{\log \frac{1}{\delta}}{n} \right)^{1/3} \quad \text{since $r \le \sigma$}\\
        &\gg O\left(\sigma \cdot   \sqrt{\frac{\xi\log \frac{\xi}{\delta}}{n}}\right) \quad \text{since $\frac{\log \frac{1}{\delta}}{n} \le 1$, and $\xi \ge \gamma \ge 1$}\\
        &\ge |\eps|
    \end{align*}
    Also, our choice of $n$ implies that
    \begin{align*}
        \frac{n}{\xi} \ge \left(\frac{\gamma^{5/12} \sigma}{r} \right)^{6 + \alpha/2} \log \frac{\xi}{\delta}
    \end{align*}
    So, we can invoke Lemma~\ref{cor:sym_kde_variance} and Lemma~\ref{lem:sym_kde_mean} to obtain the following bounds on the mean and second moment of $\wh s^{\sym}$ in Step 2, which hold with probability $1 - \frac{\delta}{\xi}$.
    \begin{align*}
        \left|\E_{x \sim f_r}\left[\wh s_r^\sym(x) \right] - \eps \I_r \right| \lesssim \frac{\eps \I_r}{\sqrt{\gamma}} \quad \text{and} \quad \left|\E_{x \sim f_r}\left[\wh s_r^\sym(x)^2 \right] - \I_r \right|\lesssim \frac{\I_r}{\sqrt{\gamma}}
    \end{align*}
    Also, for $\wh s_r^\sym$ that satisfies the above, since $\wh s_r^\sym$ is clipped at $\frac{2}{r} \sqrt{\log \frac{n}{\xi \log \frac{\xi}{\delta}}}$, and $\frac{n}{\xi} \ge \left(\frac{\gamma^{5/12} \sigma}{r} \right)^{6+\alpha/2} \log \frac{\xi}{\delta}$, the assumptions of Lemma~\ref{lem:fisher_information_estimation} are satisfied for $N = \frac{n}{\xi}$ and failure probability $\frac{\delta}{\xi}$. So conditioned on the success of Step 1, by a union bound, with probability $1 - \frac{2\delta}{\xi}$, Property~\ref{property:local_estimation_empirical_mean} is satisfied, and simultaneously, by Lemma~\ref{lem:fisher_information_estimation}, the Fisher information estimate $\wh \I_r$ in Step 2 satisfies
    \[
    \left|\wh \I_r - \I_r \right| \lesssim \frac{\I_r}{\sqrt{\gamma}}
    \]
    Conditioned on the above, since Property~\ref{property:local_estimation_empirical_mean} is satisfied, by Lemma~\ref{thm:local_estimation}, our final estimate $\hat \mu$ satisfies that for $n' = n\left(1 - \frac{2}{\xi} \right)$ and $\delta' = \delta\left(1 - \frac{3}{\xi} \right)$, with probability $1 - \delta'$,
    
    \begin{align*}
        |\hat \mu - \mu|  &\le \left(1 + O\left(\frac{1}{\sqrt{\gamma}}\right) \right)\sqrt{\frac{2 \log \frac{2}{\delta'}}{n' \I_r}}  + O\left(\frac{ \sqrt{\log \frac{n'}{\xi\log \frac{\xi}{\delta'}}}}{r\I_r} \cdot \frac{\log \frac{2}{\delta'}}{n'}\right)+ O\left(\frac{\eps}{\sqrt{\gamma}}\right)\\
        &\le \left(1 + O\left( \frac{1}{\sqrt{\gamma}}\right) + O\left( \frac{1}{\xi}\right) \right) \sqrt{\frac{2\log \frac{2}{\delta}}{n \I_r}} +  O\left(\frac{ \sqrt{\log \frac{n}{\xi\log \frac{\xi}{\delta}}}}{r\I_r} \cdot \frac{\log \frac{2}{\delta}}{n}\right)+ O\left(\frac{\eps}{\sqrt{\gamma}}\right)
    \end{align*}
    Now, we bound $\frac{\eps}{\sqrt{\gamma}}$. We have,
    \begin{align*}
        \frac{\eps}{\sqrt{\gamma}} &\lesssim \frac{\sigma \sqrt{\I_r}}{\sqrt{\gamma}}  \sqrt{\frac{\xi \log \frac{\xi}{\delta}}{n \I_r}}\\
        &\le \frac{\sigma}{r \sqrt{\gamma}}\sqrt{\frac{\xi \log \frac{\xi}{\delta}}{n \I_r}} \quad \text{since $\I_r \le \frac{1}{r^2}$}
    \end{align*}
    Finally, we bound $\frac{\sqrt{\log \frac{n}{\xi\log \frac{\xi}{\delta}}} \log \frac{1}{\delta}}{nr \I_r}$.
    First note that
    \begin{align*}
        r^2 &\ge \sigma^2 \gamma^{5/6} \xi^{1/3} \left(\frac{\log \frac{1}{\delta}}{n} \right)^{1/3}\\
        &\gtrsim \frac{1}{\I_r} \gamma^{5/6} \xi^{1/3} \left(\frac{\log \frac{1}{\delta}}{n} \right)^{1/3} \quad \text{since $\I_r \ge \frac{1}{\sigma^2 + r^2} \gtrsim \frac{1}{\sigma^2} $ since $r \le \sigma$}\\
    \end{align*}
    So, we have
    \begin{align*}
        \frac{\sqrt{\log \frac{n}{\xi\log \frac{\xi}{\delta}}} \log \frac{1}{\delta}}{r\I_r n} &\le \frac{1}{r\I_r} \left(\frac{\log \frac{1}{\delta}}{n} \right)^{1 - \alpha} \quad \text{since $\frac{n}{\log\frac{1}{\delta}} \ge 1$}\\
        &\lesssim \frac{1}{\sqrt{\I_r} \gamma^{5/12} \xi^{1/3}} \left(\frac{\log \frac{1}{\delta}}{n} \right)^{5/6 - \alpha}\\
        &= \frac{\gamma^{1/12}}{\sqrt{\I_r \gamma} \xi^{1/3}} \left(\frac{\log \frac{1}{\delta}}{n} \right)^{5/6 - \alpha}\\
        &\lesssim \frac{1}{\sqrt{\I_r \gamma} \xi^{1/3}} \left(\frac{\log \frac{1}{\delta}}{n} \right)^{4/5 - \alpha} \quad \text{since $\gamma \le \left(\frac{n}{\xi \log \frac{1}{\delta}} \right)^{2/5 - \alpha}\le \left(\frac{n}{\log \frac{1}{\delta}}\right)^{2/5}$}\\
        &\lesssim \frac{1}{\sqrt{\gamma}} \sqrt{\frac{\log \frac{1}{\delta}}{n\I_r}}
    \end{align*}
    So we have shown that 
    \[
        |\hat \mu - \mu| \le \left(1 + O\left( \frac{1}{\sqrt{\gamma}}\right) + O\left( \frac{1}{\xi}\right) \right) \sqrt{\frac{2\log \frac{2}{\delta}}{n \I_r}}+ \frac{\sigma}{r \sqrt{\gamma}}\sqrt{\frac{\xi \log \frac{\xi}{\delta}}{n \I_r}}
    \]
    Thus, by a union bound, with probability $1 - \delta$ in total, the claim follows.
\end{proof}

\maintheorem*
\begin{proof}
    Set $\xi = \frac{1}{\eta}$, and let $\gamma = \frac{1}{\eta^2}$. First we check the conditions of Lemma~\ref{lem:global_estimation_empirical_mean} that
    \begin{itemize}
        \item $\xi > C_1$ for sufficiently large constant $C_1$
        \item $\xi \le \gamma \le \left(\frac{n}{\xi \log \frac{1}{\delta}} \right)^{2/5 - \alpha}$ for constant $\alpha>0$
        \item $\frac{n}{\log \frac{1}{\delta}} \ge \xi \left( \frac{\gamma^{5/12} \sigma}{r}\right)^{6+\alpha}$ for constant $\alpha > 0$
    \end{itemize}
    We have, 
    \[
        \xi = \frac{1}{\eta} = \left(\frac{n}{\log \frac{1}{\delta}}\right)^{\frac{1}{13} - \beta} \ge C^{\frac{1}{13} - \beta} \ge C_1
    \]
    for large enough constant $C$, by assumption. We also have,
    \begin{align*}
        \gamma = \frac{1}{\eta^2} \ge \frac{1}{\eta} = \xi
    \end{align*}
    and
    \begin{align*}
        \gamma = \frac{1}{\eta^2} = \left(\frac{ n}{\log \frac{1}{\delta}} \right)^{\frac{2}{13} - 2 \beta} = \eta^{\frac{2}{5} - \beta} \left(\frac{n}{\log \frac{1}{\delta}} \right)^{\frac{8}{65} - \Omega(\beta)}\le \left(\frac{n}{\xi \log \frac{1}{\delta}} \right)^{2/5-\alpha}
    \end{align*}
    Finally, we have
    \begin{align*}
        \xi\left(\frac{\gamma^{5/12} \sigma}{r} \right)^{6+\alpha} &\le \frac{1}{\eta}\left(\frac{1}{\eta^{11/6}}\right)^{6+\alpha} \quad \text{since $\xi = \frac{1}{\eta}$, $\gamma = \frac{1}{\eta^2}$, and $r \ge \eta \sigma$}\\
        &= \frac{1}{\eta^{12 + \frac{11}{6}\alpha}}\\
        &\le \frac{n}{\log \frac{1}{\delta}} \quad \text{since $\eta = \left(\frac{\log \frac{1}{\delta}}{n} \right)^{1/13}$}
    \end{align*}
    So we have verified the above conditions, which, along with the fact that $r \le \sigma$, allow us to invoke Lemma~\ref{lem:global_estimation_empirical_mean}. The Lemma gives that with probability $1 - \delta$,
    \begin{align*}
        |\wh \mu - \mu| \le \left(1 + O\left( \frac{1}{\sqrt{\gamma}}\right) + O\left( \frac{1}{\xi}\right) \right) \sqrt{\frac{2\log \frac{2}{\delta}}{n \I_r}}+ \frac{\sigma}{r \sqrt{\gamma}}\sqrt{\frac{\xi \log \frac{\xi}{\delta}}{n \I_r}}
    \end{align*}
    We bound the last term above.
    \begin{align*}
        \frac{\sigma}{r \sqrt{\gamma}} \sqrt{\frac{\xi \log \frac{\xi}{\delta}}{n \I_r}} &\le \frac{\eta}{\sqrt{\gamma}}\sqrt{\frac{\xi \log \frac{\xi}{\delta}}{n \I_r}} \quad \text{since $r \ge \eta \sigma$}\\
        &= \eta^2 \sqrt{\frac{\xi \log \frac{\xi}{\delta}}{n \I_r}} \quad \text{substituting $\gamma$}\\
        &\le \eta \sqrt{\frac{\log \frac{1}{\delta}}{n}} \quad \text{since $\xi = \frac{1}{\eta}$}
    \end{align*}
    Along with the above and the setting of $\gamma$ and $\xi$, we have
    \[
        |\wh \mu - \mu| \le \left(1 + O(\eta) \right) \sqrt{\frac{2 \log \frac{2}{\delta}}{n \I_r}}
    \]
    Reparametrizing $\eta$ gives the claim.
\end{proof}

%% file: misc_lemmas.tex
\section{Properties of $r$-smoothed distributions}
\subsection{Score bound in terms of density}
The next lemma shows that $s_r$ is bounded in terms of $f_r$ and $r$.

\begin{lemma}
\label{lem:generic_density_large_score_small}
Let $f^*$ be an arbitrary distribution, and let $f_r$ be the $r$-smoothed version of $f^*$. Let $s_r$ be the score function of $f_r$. We have
$$|s_r(x)| \le \frac{1}{r} \sqrt{2 \log \frac{1}{\sqrt{2 \pi} r f_r(x)}}$$
\end{lemma}
\begin{proof}
Let $w_r$ be the pdf of $\mathcal N(0, r^2)$. By definition of $r$-smoothing, we have that when $X \sim f_r$, $X = Y + Z_r$ where $Y \sim f^*$ and $Z_r \sim w_r$ for independent $Y, Z_r$. So,
\begin{align*}
    f_r(x) = \Pr_{X \sim f_r}[X = x] = \int_{-\infty}^\infty \Pr_{Y \sim f^*}[Y = y] \Pr_{Z_r \sim w_r}[Z = Y - x] dy =  \E_{Y \sim f^*}[w_r(x - Y)]
\end{align*}
So, we have
\begin{align*}
    \Pr[X = x | Y = y] = \E[w_r(x - Y) | Y = y] = w_r(x - y)
\end{align*}
Now, since $w_r(x) = \frac{1}{\sqrt{2 \pi}r} e^{-\frac{-x^2}{2 r^2}}$
$$(x - Y) = r \sqrt{2 \log \frac{1}{ \sqrt{ 2 \pi} r \cdot w_r(x - Y) }}$$
So, by Lemma~\ref{lem:shifted_score_characterization},
\begin{align*}
    s_r(x) &= \E\left[\frac{Z_r}{r^2} \Big| X = x \right]\\
    &= \frac{1}{r^2} \E\left[x - Y | X = x \right] \quad \text{since } X = Y + Z_r\\
    &= \frac{1}{r} \E\left[ \sqrt{2 \log \frac{1}{\sqrt{2 \pi} r \cdot w_r(x - Y)}} \Big| X = x\right]\\
    &= \frac{1}{r} \int_{-\infty}^\infty \sqrt{2 \log \frac{1}{\sqrt{2 \pi} r \cdot w_r(x - y)}} \Pr[Y = y | X = x] dy\\
    &= \frac{1}{r} \int_{-\infty}^\infty \sqrt{2 \log \frac{1}{\sqrt{2 \pi} r \cdot w_r(x - y)}} \frac{\Pr[Y = y] \Pr[X = x | Y = y]}{\Pr[X = x]} dy \quad \text{(by Bayes' Theorem)}\\
    &= \frac{1}{r} \int_{-\infty}^\infty \frac{w_r(x - y)}{f_r(x)} \sqrt{2 \log \frac{1}{\sqrt{2 \pi} r \cdot w_r(x - y)}} \Pr[Y = y] dy\\
    &= \frac{1}{r} \E\left[\frac{w_r(x - Y)}{f_r(x)} \sqrt{2 \log \frac{1}{\sqrt{2 \pi} r \cdot w_r(x - Y)} }\right]
\end{align*}

Now, $g(z) = z \sqrt{2 \log \frac{1}{\sqrt{2 \pi} r \cdot z}} $ is concave on $[0, 1]$. So, by Jensen's inequality,
\begin{align*}
    s_r(x) &\le \frac{\E_{Y \sim f^*}[w_r(x - Y)]}{rf_r(x)} \sqrt{2 \log \frac{1}{\sqrt{2 \pi} r \cdot \E[w_r(x - Y)]}}& \\
    &= \frac{1}{r} \sqrt{2 \log \frac{1}{\sqrt{2 \pi} r \cdot f_r(x)}} \quad \text{since } f_r(x) = \E[w_r(x - Y)]\\
\end{align*}
\end{proof}

\subsection{$r$-smoothed score is $O(1/r)$-subgaussian}
The next two lemmas together show that the score function of an $r$-smoothed distribution is $O\left( \frac{1}{r}\right)$-subgaussian.
\begin{lemma}
\label{lem:FirstMomentMixturePortion}
Consider the distribution $f_r$ which is the $r$-smoothed version of distribution $f$.
That is, $f_r$ has density $f_r(x) = \E_{y \from h}[\frac{1}{\sqrt{2\pi r^2}} e^{-\frac{(x-y)^2}{2r^2}}]$.
Then, with probability at least $1-(1+\tau)\delta$, 	we sample a point $x \sim f_r$ such that
\begin{align*}  	
&\E_{y \sim f}\left[\1\left[(x-y)^2 > 2r^2\log\frac{1}{\delta}\right] \frac{|x-y|}{r^2} \frac{1}{\sqrt{2\pi r^2}} e^{-\frac{(x-y)^2}{2r^2}}\right]\\
\le\;&\frac{1}{\tau} \E_{y \sim f}\left[\1\left[(x-y)^2 \le 2r^2\log\frac{1}{\delta}\right] \frac{|x-y|}{r^2} \frac{1}{\sqrt{2\pi r^2}} e^{-\frac{(x-y)^2}{2r^2}}\right]
\end{align*}
\end{lemma}

\begin{proof}
	Observe that, at any point $x$ violating the above inequality, we have
	\begin{align*}
		f_r(x) &= \E_{y \from f}\left[\1\left[(x-y)^2 \le 2r^2\log\frac{1}{\delta}\right] \frac{|x-y|}{r^2} \frac{1}{\sqrt{2\pi r^2}} e^{-\frac{(x-y)^2}{2r^2}}\right]\\
		&\;\;+ \E_{y \from f}\left[\1\left[(x-y)^2 > 2r^2\log\frac{1}{\delta}\right] \frac{|x-y|}{r^2} \frac{1}{\sqrt{2\pi r^2}} e^{-\frac{(x-y)^2}{2r^2}}\right]\\
		&\le (1+\tau)\E_{y \from f}\left[\1\left[(x-y)^2 > 2r^2\log\frac{1}{\delta}\right] \frac{|x-y|}{r^2}  \frac{1}{\sqrt{2\pi r^2}} e^{-\frac{(x-y)^2}{2r^2}}\right]
	\end{align*}
    We wish to bound the probability of sampling $x$ violating the lemma inequality, which is bounded by the integral of the above right hand side.
    We can further bound it using the following:
	\begin{align*}
		&\int \E_{y \from f}\left[\1\left[(x-y)^2 > 2r^2\log\frac{1}{\delta}\right] \frac{|x-y|}{r^2}  \frac{1}{\sqrt{2\pi r^2}} e^{-\frac{(x-y)^2}{2r^2}}\right] \, \d x\\
		=& \E_{y \from f}\left[\int \1\left[(x-y)^2 > 2r^2\log\frac{1}{\delta}\right] \frac{|x-y|}{r^2}  \frac{1}{\sqrt{2\pi r^2}} e^{-\frac{(x-y)^2}{2r^2}}\, \d x\right]\\
		\le& \E_{y\from f}[\delta] = \delta
	\end{align*}
	
	Thus the probability of sampling a point $x$ violating the lemma inequality is upper bounded by the integral of $f_r(x)$ over those points, which is in turn upper bounded by $(1+\tau)\delta$.
\end{proof}
\begin{lemma}[Score is $O(1/r)$-subgaussian]
\label{lem:score-subgaussian}
Let $s_r$ be the score function of an $r$-smoothed distribution $f_r$. We have that for $x \sim f_r$, with probability $1-\delta$, $\abs{s_r(x)} \lesssim \frac{1}{r} \sqrt{\log \frac{2}{\delta}}$.
\end{lemma}
\begin{proof}

By Lemma~\ref{lem:FirstMomentMixturePortion} using $\tau = 1$, with probability $1-2\delta$ over sampling a single point $x \from f_r$, the point $x$ satisfies 
	\begin{align*}  	
&\E_{y \from f}\left[\1\left[(x-y)^2 > 2r^2\log\frac{1}{\delta}\right] \frac{|x-y|}{r^2}  \frac{1}{\sqrt{2\pi r^2}} e^{-\frac{(x-y)^2}{2r^2}}\right]\\
\le\;& \E_{y \from f}\left[\1\left[(x-y)^2 \le 2r^2\log\frac{1}{\delta}\right] \frac{|x-y|}{r^2}  \frac{1}{\sqrt{2\pi r^2}} e^{-\frac{(x-y)^2}{2r^2}}\right]
	\end{align*}
	
	And so,
	\begin{align*}
    s_r(x) = \frac{f'_r(x)}{f_r(x)}
    =\;& \frac{\E_{y \from f}\left[\frac{y-x}{r^2}\frac{1}{\sqrt{2\pi r^2}}e^{-\frac{(x-y)^2}{2r^2}}\right]}{f_r(x)}\\
	\le\;&\frac{\E_{y \from f}\left[\frac{|x-y|}{r^2}\frac{1}{\sqrt{2\pi r^2}}e^{-\frac{(x-y)^2}{2r^2}}\right]}{f_r(x)}\\
	=\;& \frac{\E_{y \from f}\left[\1\left[(x-y)^2 \le 2r^2\log\frac{1}{\delta}\right] \frac{|x-y|}{r^2}  \frac{1}{\sqrt{2\pi r^2}} e^{-\frac{(x-y)^2}{2r^2}}\right]}{f_r(x)}\\
	+& \frac{\E_{y \from f}\left[\1\left[(x-y)^2 > 2r^2\log\frac{1}{\delta}\right] \frac{|x-y|}{r^2}  \frac{1}{\sqrt{2\pi r^2}} e^{-\frac{(x-y)^2}{2r^2}}\right]}{f_r(x)}\\
	\le\;& 2\cdot \frac{\E_{y \from f}\left[\1\left[(x-y)^2 \le 2r^2\log\frac{1}{\delta}\right] \frac{|x-y|}{r^2}  \frac{1}{\sqrt{2\pi r^2}} e^{-\frac{(x-y)^2}{2r^2}}\right]}{f_r(x)}\\
	\le\;& \frac{2\sqrt{2}}{r}\sqrt{\log\frac{1}{\delta}}\frac{\E_{y \from f}\left[\1\left[(x-y)^2 \le 2r^2\log\frac{1}{\delta}\right] \frac{1}{\sqrt{2\pi r^2}} e^{-\frac{(x-y)^2}{2r^2}}\right]}{f_r(x)}\\
	\le\;& \frac{2\sqrt{2}}{r}\sqrt{\log\frac{1}{\delta}}
	\end{align*}
 Reparameterizing from $2\delta$ to $\delta$ gives the lemma result.
\end{proof}
\subsection{Lemmas from \cite{finite_sample_mle}}
Here, we recall some of the properties of $r$-smoothed distributions shown in \cite{finite_sample_mle}
\begin{lemma}[From \cite{finite_sample_mle}]
\label{lem:shifted_score_characterization}
    Let $s_r$ be the score function of $r$-smoothed distribution $f_r$. Then,
    \begin{align*}
        \frac{f_r(x+\eps)}{f_r(x)} = \E_{Z_r | x} \left[e^{\frac{2 \eps Z_r - \eps^2}{2 r^2}} \right] \quad \text{and in particular} \quad s_r(x) = \frac{1}{r^2}\E_{Z_r | x} [Z_r]
    \end{align*}
    and
    \begin{align*}
        s_r(x+\eps) = \frac{\E_{Z_r|x}\left[e^{\frac{\eps Z_r }{r^2}} \frac{Z_r - \eps}{r^2} \right]}{\E_{Z_r | x} \left[e^{\frac{\eps Z_r }{r^2}} \right]}
    \end{align*}
    
\end{lemma}
\begin{lemma}[From \cite{finite_sample_mle}]
\label{lem:shifted_score_expectation}
    Let $s_r$ be the score function of an $r$-smoothed distribution $f_r$ with Fisher information $\I_r$. Then for any $|\eps| \le r/2$,
    \[
        \E_{x \sim f_r}[s_r(x+\eps)] = -\I_r \eps + \Theta\left(\sqrt{\I_r} \frac{\eps^2}{r^2} \right)
    \]
\end{lemma}
\begin{lemma}[From \cite{finite_sample_mle}]
\label{lem:shifted_score_second_moment}
    Let $s_r$ be the score function of an $r$-smoothed distribution $f_r$ with Fisher information $\I_r$. Then, for any $|\eps| \le r/2$,
    \[
        \E_{x \sim f_r} [s_r^2(x+\eps)] \le \I_r + O\left(\frac{\eps}{r} \I_r \sqrt{\log \frac{1}{r^2 \I_r}} \right) 
    \]
\end{lemma}
\begin{lemma}[From \cite{finite_sample_mle}]
\label{lem:I_atmost_1/r^2}
    Let $\I_r$ be the Fisher information of an $r$-smoothed distribution $f_r$. Then $\I_r \le 1/r^2$.
\end{lemma}
\begin{lemma}[From \cite{finite_sample_mle}]
    \label{lem:Delta_eps_bound}
    Let $f^*$ be an arbitrary distribution, and let $f_r$ be the $r$-smoothed version of $f^*$. Define
    \[
        \Delta_\eps(x) := \frac{f_r(x + \eps) - f_r(x) - \eps f_r'(x)}{f_r(x)}
    \]
    Then, for any $|\eps| \le r/2 $,
    \[
        \E_{x \sim f_r}\left[\Delta_\eps(x)^2 \right] \lesssim \frac{\eps^4}{r^4}
    \]
\end{lemma}

%% file: median_of_pairwise_means.tex
\section{Median of Pairwise Means Estimator}
Using results in~\cite{mintonprice}, we show that the median of pairwise means is a good estimator for symmetric random variables.  In particular, it matches the convergence of the median-of-means estimator for all $(\eps, \delta)$ without needing to specify $\eps$ and $\delta$.
\medianofmeanslemma*
\begin{proof}
   Let $Y_i = \frac{1}{2}(X_{2i -1} + X_{2i})$ for $i \in [n/2]$.  Let $p$ be the pdf of $X - \mu$, and $q$ be the pdf of $Y - \mu$.  Since $X-\mu$ is symmetric about $0$, the Fourier transform $\wh{p}$ of $p$ is real-valued.  By the Fourier convolution theorem, $q$ has nonnegative Fourier transform.  Then by Lemma~3.1 of~\cite{mintonprice}, for any $\eps < 1$,
   \[
   \Pr[\abs{Y_i} < \eps\sigma] \geq C_3\eps
   \]
   for a universal constant $C_3$. Then it is easy to show (e.g., Lemma~3.3 of~\cite{mintonprice}):
   \[
   \Pr[\abs{\wh{\mu} - \mu} > \eps \sigma] \leq 2 e^{-\frac{C_3^2}{4} \eps^2 n}.
   \]
   Setting $\eps = \frac{2}{C_3}\sqrt{\frac{\log \frac{2}{\delta}}{n}}$ gives the result, as long as $n > \frac{4}{C_3^2}\log \frac{2}{\delta}$ so this $\eps < 1$.

   There's a remaining regime of $\frac{C_3^2}{4} \leq \frac{\log \frac{1}{\delta}}{n} \leq C_1$ for which we need to prove a $\Theta(\sigma)$ bound on $\abs{\wh{\mu} - \mu}$.  Note that $Y_i$ has variance $\sigma^2/2$, so for any $a > 0$, with probability $1-a$ we have $\abs{Y_i-\mu} \leq \frac{\sigma}{\sqrt{2a}}$.  Let $E_i$ denote the event that $\abs{Y_i-\mu} > \frac{\sigma}{\sqrt{2a}}$.  Then
   \[
   \Pr[\abs{\wh{\mu} - \mu} > \frac{\sigma}{\sqrt{2a}}] \leq \Pr[\sum_{i=1}^{n/2} E_i \geq \frac{n}{4}] \leq \binom{n/2}{n/4} a^{n/4} \leq (4a)^{n/4}.
   \]
   which is $\delta$ for $a = \frac{1}{4}e^{-\frac{4}{n}\log \frac{1}{\delta}} \geq \frac{1}{4} e^{-4 C_1}$.  Thus with probability $1-\delta$, $\abs{\wh{\mu} - \mu} \leq \sqrt{2} e^{2 C_1} \sigma \eqsim \sigma$.
\end{proof}